\documentclass[11pt, oneside]{article}   	

\usepackage[a4paper, left=2.5cm, right=2.5cm, top=2.5cm]{geometry}
\usepackage{graphicx}				
\usepackage{caption}
\usepackage{subcaption}							
\usepackage{pdfpages}
\usepackage{amssymb}
\usepackage{amsmath}
\usepackage{amsthm}
\usepackage[english]{babel}
\usepackage[utf8]{inputenc}
\usepackage[T1]{fontenc}
\usepackage{lmodern}
\usepackage{float}
\usepackage{framed, color}
\usepackage{cite}
\usepackage{comment}
\usepackage{enumitem}
\usepackage{abstract}
\usepackage{hyperref}

\newtheorem{theorem}{Theorem}[section]
\newtheorem{corollary}[theorem]{Corollary}
\newtheorem{lemma}[theorem]{Lemma}

\newtheorem{example}[theorem]{Example}
\theoremstyle{definition}
\newtheorem{remark}[theorem]{Remark}
\newtheorem{definition}[theorem]{Definition}
\newtheorem{algorithm}[theorem]{Algorithm}

\newtheorem{assumption}{Assumption}

\DeclareMathOperator*{\argmin}{arg\,min}
\newcommand{\G}{\mathcal{G}}
\newcommand{\R}{\mathbb{R}}
\newcommand{\N}{\mathbb{N}}
\DeclareMathOperator{\prox}{prox}
\DeclareMathOperator{\epi}{epi}
\DeclareMathOperator{\sign}{sign}
\DeclareMathOperator{\conv}{conv}
\DeclareMathOperator{\supp}{supp}
\DeclareMathOperator{\dom}{dom}
\DeclareMathOperator{\gph}{gph}
\def\dx{\,\mathrm{d}x}
\def\dt{\,\mathrm{d}t}

\numberwithin{equation}{section}

\title{A proximal gradient method for control problems with nonsmooth and nonconvex control cost}
\author{Carolin Natemeyer, Daniel Wachsmuth
\footnote{Institut f\"ur Mathematik,
Universit\"at W\"urzburg,
97074 W\"urzburg, Germany, {\tt carolin.natemeyer@mathematik.uni-wuerzburg.de, daniel.wachsmuth@mathematik.uni-wuerzburg.de}.
This research was partially supported by the German Research Foundation DFG under project grant Wa 3626/3-2.}}

\begin{document}
	\maketitle

\paragraph{Abstract.}
We investigate the convergence of an application of a proximal gradient method
to control problems with nonsmooth and nonconvex control cost.
Here, we focus on control cost functionals that promote sparsity, which includes functionals of $L^p$-type for $p\in [0,1)$.
We prove stationarity properties of weak limit points of the method. These
properties are weaker than those provided by Pontryagin's maximum principle
and weaker than $L$-stationarity.


\section{Introduction}
	Let $\Omega\subset \R^n$ be Lebesgue measurable with finite measure.
	We consider a possibly non-smooth optimal control problem of type
	\begin{equation} \label{P} \tag{P}
	\min\limits_{u\in L^2(\Omega)} f(u)+ \int_\Omega g(u(x))\dx.
	\end{equation}

	Here, the function  $g:\R\to \R\cup\{+\infty\}$ is nonconvex and nonsmooth. Examples include \[g(u)=|u|^p, \quad p\in(0,1),\] and
	\[g(u) = |u|_0: = \begin{cases} 1 & \text{ if } u\ne 0\\ 0 & \text{ if } u=0.\end{cases}.\]
	The function $f:L^2(\Omega) \to \R$ is assumed to be smooth. Here, we have in mind to choose $f(u):=f(y(u))$
	as the smooth part of an optimal control problem incorporating the state equation and possibly  smooth cost functional.
	We will make the assumptions on the ingredients of the control problem precise below in Section \ref{sec2}.

	Due to the properties of $g$, the optimization problem \eqref{P} is challenging in several ways. First of all,
	the resulting integral functional $u\mapsto \int_\Omega g(u(x))\dx$ is not weakly lower semicontinuous in $L^2(\Omega)$,
	so it is impossible to prove existence of solutions of \eqref{P} by the direct method.
	Second, it is challenging to solve numerically, i.e., to compute local minima or stationary points.

	In this paper, we address this second issue.
	Here, we propose to use the proximal gradient method (also called forward-backward algorithm \cite{bauschkecombettes}).
	The main idea of this method is as follows: Suppose the objective is to minimize a sum $f+j$ of
	two functions  $f$ and $j$ on the Hilbert space $H$ where $f$ is smooth.
	Given an iterate $u_k$, the next iterate $u_{k+1}$ is computed  as
	\begin{equation}\label{eq101}
	 u_{k+1} =  \argmin\limits_{u\in H}\left( f(u_k)+\nabla f(u_k)(u-u_k)+\frac{L}{2}\|u-u_k\|^2_{H}+j(u)\right),
	\end{equation}
	where $L>0$ is a proximal parameter, and  $L^{-1}$ can be interpreted as a step-size.
	In our setting, the functional to be minimized in each step is an integral function, whose minima can be computed by minimizing
	the integrand pointwise.
	Using the so-called prox map, that is defined by
	\begin{equation}\label{ch4:prox}
	\prox_{\gamma j}(z) = \argmin\limits_{x\in H}\left( \frac{1}{2}\|x-z\|^2_H+\gamma j(x)\right),
	\end{equation}
	where $\gamma>0$, the next iterate of the algorithm can be written as
	\[
	 u_{k+1} = \prox_{L^{-1} j}\left(u_k - \frac1L  \nabla f(u_k)\right).
	\]
	If $j\equiv0$, the method reduces to the steepest descent method. If $j$ is the indicator function
	of a convex set, then the method is a gradient projection method.
	If $f$ and $j$ are convex, then the convergence properties of the method are well-known:
	under mild assumptions
	the iterates $(u_k)$ converge weakly to a global minimum of $f+j$, see, e.g., \cite[Corollary 27.9]{bauschkecombettes}.
	If $f$ is non-convex, then weak sequential limit points of $(u_k)$ are stationary, that is, they satisfy $-\nabla f(u^*) \in \partial j(u^*)$.
	If in addition $j$ is nonconvex, then much less can be proven.
	In finite-dimensional problems, limit points are fixed points of the iteration, and satisfy the so-called $L$-stationary type conditions,
	see \cite{BeckEldar2013} and \cite[Chapter 10]{Beck2014} for optimization problems with $l^0$-constraints.
	A feasible point $u^*$ is called $L$-stationary if
	\[
	 u^* = \prox_{L^{-1} j}\left(u^* - \frac1L  \nabla f(u^*)\right).
	\]
	In a recent contribution \cite{DW:1}, the method was analyzed when applied to control problems with $L^0$-control cost.
	There it was proven that weak sequential limit points of the iterates in $L^2(\Omega)$ satisfy the $L$-stationary type condition.
	An essential ingredient of the analysis in \cite{DW:1} was that the functional $g$ is sparsity promoting: solutions of the
	proximal step are either zero or have a positive distance to zero.
	We will show how this property can be obtained under weak assumptions on the functional $g$ in \eqref{P} near $u=0$, see Section \ref{sec3}.
	Still this is not enough to conclude $L$-stationarity of limit points. We will show that weak
	limit points satisfy a weaker condition in general, see Theorem \ref{thm:set_valconv}.
	Under stronger assumptions, $L$-stationarity can be obtained (Theorems \ref{theo420}, \ref{theo421}).
	Let us emphasize that, under weak assumptions,  the sequence of iterates $(u_k)$ contains weakly converging subsequences
	but is not weakly convergent in general. Pointwise a.e.\@ and strong convergence is obtained in Theorem~\ref{thm:strongconv}.
	We apply these results to $g(u)=|u|^p$, $p\in (0,1)$ in Section \ref{ch:4}.

	Interestingly, the proximal gradient method sketched above is related to algorithms based on proximal minimization
	of the Hamiltonian in control problems. These algorithms are motivated by Pontryagin's maximum principle.
	First results for smooth problems can be found in \cite{SakawaShindo1980}. There, stationarity of pointwise limits of $(u_k)$ was proven.
	Under weaker conditions it was proved in \cite{Bonnans1986} that the residual in the optimality conditions tends to zero.
	These results were transferred to control problems with parabolic partial differential equations in \cite{BreitenbachBorz2019a}.


\paragraph{Notation.}
We will frequently use $\bar \R:= \R \cup \{+\infty\}$.

\section{Preliminary considerations}\label{sec2}

Throughout the paper, we will use the following assumption on the function $f$.
\begin{assumption}\label{ass_A}
The functional $f:L^2(\Omega)\to \R$ is bounded from below and weakly lower semicontinuous. Moreover, $f$ is Fréchet differentiable and $\nabla f:L^2(\Omega)\to L^2(\Omega)$ is Lipschitz continuous with constant $L_f$, i.e.,
\[
\|\nabla f(u_1)-\nabla f(u_2)\|_{L^2(\Omega)}\leq L_f\|u_1-u_2\|_{L^2(\Omega)}
\]
holds for all  $u_1,u_2\in L^2(\Omega)$.
\end{assumption}
For the moment, let $g:\R \to \bar\R$ be lower semicontinuous and bounded from below. In Section \ref{sec3} below, we will
give the precise assumptions on $g$ that allow sparse controls.
Let $u\in L^2(\Omega)$ be given. Then $x\mapsto g(u(x))$ is a measurable function, and we define
\[
 j(u):= \int_\Omega g(u(x))\dx.
\]
Then $j:L^2(\Omega) \to \bar \R$ is well-defined and lower semicontinuous, but not weakly lower semicontinuous in general.
Hence standard existence proofs cannot be applied. For a discussion, we refer to \cite{ItoKunisch:1,DW:1}

\begin{remark}
 The results are also valid for the general case that $g$ depends on $x\in \Omega$, which results in the integral functional
 $j(u)=\int_\Omega g(x,u(x))\dx$,
 provided $g:\Omega\times \R\to\bar \R$ is a normal integrand, for the definition we refer to \cite[Definition VIII.1.1]{EkelandTemam1999}.
\end{remark}

\color{black}

%

	\subsection{Necessary optimality conditions}
	The mapping $u\mapsto \int_\Omega g(u(x))\dx$ is not directionally differentiable  in general, and thus there is no first order optimality condition.
	In the following we are going to derive a necessary optimality condition for  \eqref{P}, known as Pontryagin maximum principle,  where no derivatives of the functional are involved.
	We formulate the Pontryagin maximum principle (PMP)  as in \cite{DW:1}. A control $\bar u\in L^2(\Omega)$ satisfies (PMP) if and only if
	for almost all $x\in\Omega$
	\begin{equation} \label{PMP} \nabla f(\bar u)(x) \bar u(x) + g(\bar u(x))
	\leq \nabla f(\bar u)(x)\cdot v+ g(v)
	\end{equation}
	holds true for all  $v\in \R$. The following result is shown in \cite[Thm. 2.5]{DW:1} for the special choice $g(u):=|u|_0$.

	\begin{theorem}[Pontryagin maximum principle] \label{thm:PMP}
		Let $\bar u\in L^2(\Omega)\cap L^\infty(\Omega)$ be locally optimal to \textnormal{(\ref{P})}. Furthermore, assume $f$ satisfies $$f(u)-f(\bar u)=\nabla f(\bar u)\cdot(u-\bar u)+o(\|u-\bar u\|_{L^1(\Omega)})$$ for all $u\in U_{ad}$.  Then $\bar u$ satisfies the Pontryagin maximum principle \eqref{PMP}.
	\end{theorem}
	\begin{proof}
		Let $\bar u$ be a local solution to \eqref{P}. We will use needle perturbations of the optimal control.
		Let $E:=\{(v_i,t_i),\ i\in \mathbb N\}$
		be a countable dense subset of $$\epi(g) = \{(v,t)\in \R\times\R: g(v)\leq t\}.$$
		For arbitrary $x\in\Omega$ define $u_{r,i}\in L^2(\Omega)$ by
		$$u_{r,i}:=\begin{cases} v_i\enspace &\textnormal{ on } B_r(x),\\
		\bar u&\textnormal{otherwise} \end{cases}$$
		for some $r>0$ and $i\in\mathbb{N}$.
		Let $\chi_r:=\chi_{B_r(x)}$, then we have $u_{r,i}=(1-\chi_r)\bar u+\chi_r v_i$ and
		\begin{align*} \|u_{r,i}-\bar u\|_{L^1(\Omega)}&=\|\chi_r(v_i-\bar u)\|_{L^1(\Omega)} \leq (|v_i|+\|\bar u\|_{L^\infty(\Omega)})\|\chi_r\|_{L^1(\Omega)}\\&=(|v_i|+\|\bar u\|_{L^\infty(\Omega)})|B_r(x)|.
		\end{align*}
		With $j(u):= \displaystyle\int_\Omega g(u(x))\dx$ 
		 we get

		\begin{align*}
		0 & \leq f(u_{r,i})+j(u_{r,i})-f(\bar{u})-j(\bar{u}) \\
		& =  \int_\Omega\nabla f(\bar{u})(u_{r,i}-\bar u)\dt +o(\|u_{r,i}-\bar u\|_{L^1(\Omega)})
		+ \int_\Omega( g(u_{r,i}) -g(\bar u)) \dt\\
		&\leq \int_{B_r(x)}\nabla f(\bar u)(v_i-\bar u)+( t_i -g(\bar u))\dt +o(\|u_{r,i}-\bar u\|_{L^1(\Omega)})
		\end{align*}

		After dividing above inequality by $|B_r(x)|$ and passing $r\searrow 0$, we obtain by Lebesgue's differentiation theorem
		\begin{equation}\label{ch3:pmpinequ} 0\leq\nabla f(\bar u)(x)\cdot(v_i-\bar u(x))+(t_i-g(\bar u(x))).
		\end{equation}
		This holds for every Lebesgue point $x\in\Omega$ of the integrands, i.e., for all $x\in\Omega\setminus N_i$, where $N_i$ is a set of zero Lebesgue measure,
		on which the above inequality is not satisfied. Since the union $\bigcup_iN_i$ is also of measure zero, \eqref{ch3:pmpinequ} holds true for all $x\in \Omega\setminus\bigcup_{i}N_i$ for all $i$.
		Due to the density assumption, for $(v,t)\in \epi(g+I_{U_{ad}})$ we find a sequence $(v_k,t_k)\to(v,t)$ with $(v_k,t_k)\in E$,
		and hence for almost all $x\in\Omega$ it holds
		$$0\leq\nabla f(\bar u)(x)\cdot(v-\bar u(x))+(t-g(\bar u(x))).$$
		for all $(v,t)\in \epi(g)$. Choosing $t = g(v)$ yields the claim.
	\end{proof}

\section{Sparsity promoting proximal operators}
\label{sec3}

In this section, we will investigate the minimization problems that have to be solved in order to
compute the proximal gradient step in \eqref{eq101}.
Let $g:\R\to\bar\R$ be proper and lower-semicontinuous.
For $s>0$ and $q\in \R$,
we define the function
\[
 h_{q,s}(u) := -qu+\frac{1}{2}u^2+sg(u).
\]
Here, we have in mind to set $q:=-\nabla f(u_k)(x)$.
Let us investigate scalar-valued optimization problems of form
\begin{equation} \label{P:min_h}
\min\limits_{u\in\R} h_{q,s}(u).
\end{equation}
The solution set  is given by the proximal map $\prox_{sg}:\R\rightrightarrows\R$ of $g$,
\[
\prox_{sg}(q) := \argmin\limits_{u\in\R}\enspace \left(\frac{1}{2}|u-q|^2+sg(u)\right).
\]
If $g$ is convex then \eqref{P:min_h} is a convex problem, and the proximal map is single-valued.
If $g$ is bounded from below and lower semicontinuous, $\prox_{sg}(q)$ is nonempty for all $q$ but may be multi-valued for some $q$.

The focus of this section is to investigate under which assumptions $\prox_{sg}$ is sparsity promoting:
Here, we want to prove that there is $\sigma>0$ such that
\[
 u\in \prox_{sg} \ \Rightarrow \ u=0 \text{ or } |u|\ge \sigma.
\]
In \cite{nikolova:1}, this was also investigated for some special cases of non-convex functions.
We will show that the following assumption is enough to guarantee the sparsity promoting property,
it contains the result from \cite{nikolova:1} as a special case.
\color{black}

\begin{assumption}\label{ass_B}
\phantom{bla bla}
\begin{enumerate}[label=(B\arabic*)]
	\item\label{assb:1} $g:\R\to\bar\R$ is lower semicontinuous, symmetric with $g(0)=0$.
	\item\label{assb:2} There is $u\ne0$ such that $g(u)\in \R$.
	\item\label{assb:4} $g$ satisfies one of the following properties:
	\begin{enumerate}[label=(B\arabic{enumi}.\alph*)]
		\item\label{assb:4a} $g$ is twice differentiable on an interval $(0,\epsilon)$ for some $\epsilon>0$ and $\limsup\limits_{u\searrow0}g''(u) \in (-\infty,0)$,
		\item\label{assb:4b} $g$ is twice differentiable on an interval $(0,\epsilon)$ for some $\epsilon>0$ and $\lim\limits_{u\searrow0}g''(u) =-\infty$,
		\item\label{assb:4c} $0 < \liminf_{u\searrow0} g(u)$.
	\end{enumerate}
	\item\label{assb:5}$g(u) \ge0$ for all $u\in \R$.
\end{enumerate}
\end{assumption}

By assumption  B, the function $g$ is non-convex in a neighborhood of $0$ and nonsmooth at $0$. Some examples are given below.
\begin{example} Functions satisfying assumption \ref{ass_B}.
	\begin{enumerate}[label=(\roman*)]
		\item $g(u):=|u|_0 := \begin{cases}
		1\quad &u\not=0,\\ 0&\text{else},
		\end{cases}$
		\item $g(u) := |u|^p,\quad p\in(0,1)$,
		\item $g(u) := \ln(1+\alpha|u|)$, with a given positive constant  $\alpha$.
		\item The indicator function of the integers $g(u):=\delta_\mathbb{Z}(u) = \begin{cases}
			0\quad&\text{if } u\in\mathbb{Z},\\
			\infty &\text{otherwise.}
		\end{cases}$
	\end{enumerate}
\end{example}

We are interested in the characterization of global solutions to \eqref{P:min_h} in terms of $q$. It is well-known that for given $s>0$ the proximal map $q\rightrightarrows \prox_{sg}(q)$ is monotone,
i.e., the inequality
$$0\leq (q_1-q_2)\left(\prox_{sg}(q_1)-\prox_{sg}(q_2)\right)$$
is satisfied for all $q_1,q_2\in\R$.
In addition, the graph of $\prox_{sg}$ is a closed set.
Moreover, the following results hold true.

\begin{lemma}\label{lem:prox_sign}
Let $g:\R\to\bar\R$ satisfy Assumption \ref{assb:1}.
Let $u\in \prox_{sg}(q)$. Then $u\ge 0$ if and only if $q\ge0$.
\end{lemma}
\begin{proof}
Due to \ref{assb:1}, we have $u \in  \prox_{sg}(q)$ if and only if $-u\in \prox_{sg}(-q)$.
The claim now follows from the monotonicity of the $\prox$-mapping.
\end{proof}

\begin{lemma}\label{lem_prox_growth}
Let $g:\R\to\bar\R$ satisfy Assumptions \ref{assb:1}, \ref{assb:5}.
Then the growth condition
\[
 |u| \le 2|q| \quad \forall u \in \prox_{sg}(q)
\]
is satisfied.
\end{lemma}
\begin{proof}
 Let $u \in \prox_{sg}(q)$. By optimality, the following inequality
 \[
 \frac12 u^2 - qu + g(u) \le g(0) = 0.
 \]
 is true. Since $g(u)\ge 0$, the claim follows.
\end{proof}
\color{black}

\begin{lemma} \label{lem:char_prox}
Let $H$ be a Hilbert space.
Let $f:H\to \bar\R$ be a function with $f(0) \in \R$. Then
$0\in \prox_f(q)$ for all $q\in H$ if and only if $f$ is of the form $f(x) = f(0) + \delta_{\{0\}}(x)$.
Here, $\delta_{\{0\}}$ is the indicator function of $\{0\}$ defined by $\delta_{\{0\}}(0)=0$ and $\delta_{\{0\}}(x)=+\infty$ for all $x\ne0$.
\end{lemma}
\begin{proof}
If $f$ is of the claimed form, then clearly $\prox_f(q) = \{0\}$ for all $q$. Now, let $0\in\prox_f(q)$ for all $q\in H$. Then it holds
\[
\frac12\|u-q\|_H^2 + f(u) \ge \frac12\|q\|_H^2 + f(0) \quad \forall u,q\in H.
\]
This is equivalent to
\[
 f(u) + \frac12\|u\|^2_H  \ge f(0) + (u,q)_H \quad \forall u,q\in H.
\]
Setting $q:=tu$ and letting $t \to +\infty$ shows $f(u)=+\infty$ for all $u\ne0$.
\end{proof}

\begin{lemma} \label{lem:ass_q0}
Let $g:\R\to\bar\R$ satisfy Assumption \ref{assb:1}.
	Let $s>0$.
	Assume there is $q_0\ge 0$ such that
	\begin{equation} \label{ass_q0}
	q_0 |u| \leq \frac{1}{2} u^2 + sg(u) \quad \forall u\in \R.
	\end{equation}
	Then $u=0$ is a global solution to \eqref{P:min_h} if
	$|q| \le q_0$.
If $|q| < q_0$ then $u=0$ is the unique global solution to \eqref{P:min_h}.
Moreover, if
\[
q_0:=\sup\{q\geq0: q|u|\leq\frac{1}{2} u^2 + sg(u) \quad \forall u\in \R\},
\]
then $|q|\leq {q_0}$ is also necessary for $u=0$ being a global solution to \eqref{P:min_h}.
\end{lemma}
\begin{proof}
	Let $ |q| \le q_0$. Take $u\ne0$, then we have
	\[
	h_{q,s}(u) = \frac12 u^2 + sg(u) -uq \ge\frac12 u^2 + sg(u) -|u| \cdot |q| \ge \frac12 u^2 + sg(u) - q_0|u|\ge 0 = h_{q,s}(0).
	\]

Note that the second inequality is strict if $|q|<q_0$.
For the second claim assume $u=0$ is a global solution to \eqref{P:min_h}.
Assume $q>0$. Then it holds
\[
 qu \le \frac12u^2 + sg(u) \quad \forall u\ge0.
\]
Since $g(u)=g(-u)$, this implies
\[
 q|u| \le \frac12u^2 + sg(u) \quad \forall u\in \R.
\]
By the definition of $q_0$, the inequality $q\le q_0$ follows. Similarly, one can prove $|q|\le q_0$ for negative $q$.
\end{proof}

 Together with Assumption \ref{ass_B}, these results allows us to show the following key observation concerning the characterization of solutions to $\eqref{P:min_h}$. A similar statement to the following can be found in {\cite[Theorem 1.1]{nikolova:1}}.

\begin{theorem} \label{thm:discontinuity}
	Let $g:\R\to\bar\R$ satisfy Assumption \ref{ass_B}.
	Then there exists $s_0\ge 0$ such that for every $s> s_0$ there is $u_0(s)>0$ such that for all $q\in\R$
	a global minimizer $u$ of \eqref{P:min_h} satisfies
	\[
	u = 0 \textnormal{ or } |u|\geq u_0(s).
	\]
	In case $g$ satisfies \ref{assb:4b} or \ref{assb:4c}, $s_0$ can be chosen to be zero.
	Moreover, for all $s>0$ there is $q_0:=q_0(s)>0$ such that $u=0$ is a global solution to \eqref{P:min_h} if and only if $|q| \leq q_0$.
If $|q| < q_0$ then $u=0$ is the unique global solution to \eqref{P:min_h}.
\end{theorem}
\begin{proof}
	Assume that the first claim does not hold. Then there are sequences $(u_n)$ and $(q_n)$ and $s>0$ with $u_n \in \prox_{sg}(q_n)$ and $u_n \to 0$. W.l.o.g., $(u_n)$ is
	a monotonically decreasing sequence of positive numbers, and hence $(q_n)$ is monotonically decreasing and non-negative by Lemma \ref{lem:prox_sign}.
	Let $u$ and $q$ denote the limits of both sequences.
	Since $u_n\ne0$ is a global minimum of $h_{q_n,s}$, it follows $h_{q_n,s}(u_n)\le h_{q_n,s}(0)=0$.
	Passing to the limit in this inequality, we obtain $\liminf_{n\to\infty} h_{q_n,s}(u_n)\le 0$, which implies
	\[
	\liminf_{n\to\infty} g(u_n)\le 0.
	\]
	With $g(0)=0$ by \ref{assb:1}, this contradicts \ref{assb:4c}.
	Let now  \ref{assb:4a} or \ref{assb:4b} be satisfied.
	Then for $n$ sufficiently large the necessary second-order optimality condition 
	$h_{q_n,s}''(u_n)\ge0$ holds, and
	we obtain
	\[
	  \limsup_{n\to\infty} h_{q_n,s}''(u_n)\ge0,
	\]
	which implies
	\[
	1+s\limsup_{n\to\infty} g''(u_n)\ge0.
	\]
	This inequality is a contradiction to \ref{assb:4a} if  $s>{-1}/{\limsup_{u\searrow0}g''(u)}>0$ and to \ref{assb:4b} for all $s$.

	By \ref{assb:1}, it holds $\prox_{sg}(q) \ne \emptyset$ for all $q$.
	Due to \ref{assb:2} and Lemma \ref{lem:char_prox}, there is $q\ge0$ such that $0\not\in \prox_{sg}$.
	The claim concerning $q_0$ follows from Assumptions \ref{assb:5}, \ref{assb:4} and Lemma \ref{lem:ass_q0}.
	First, consider that case \ref{assb:4a} or \ref{assb:4b} is satisfied, i.e.,
	there is $\epsilon_1>0$ such that $g$ is strictly concave on
	$(0,\epsilon_1]$.
	By reducing $\epsilon_1$ if necessary, we get $g(\epsilon_1)>0$.
	Since $g(u)=0$, it holds $g(u) \ge \frac{g(\epsilon_1)}{\epsilon_1}  |u| $ for all $u\in (0,\epsilon_1)$ by concavity.
	Due to symmetry, this holds for all $u$ with $|u| \le \epsilon_1$.
	Since $g(u)\ge0$ for all $u$ by \ref{assb:5}, it holds $\frac12 u^2 + sg(u) \ge \frac{\epsilon_1}2 |u|$ for all $|u| \ge \epsilon_1$.
	This proves $\frac12 u^2 + sg(u) \ge \min(\frac{\epsilon_1}2,\frac{sg(\epsilon_1)}{\epsilon_1})|u|$ for all $u$.
	Hence, the claim follows with $q_0:= \min(\frac{\epsilon_1}2,\frac{sg(\epsilon_1)}{\epsilon_1})$ by Lemma \ref{lem:ass_q0}.
	Second, if \ref{assb:4c} is satisfied, then there are $\epsilon_2,\tau>0$ such that $g(u) \ge \tau$ for all $u$ with  $|u|\in(0,\epsilon_2)$ as $g$ is lower semicontinuous.
	Therefore, it holds $g(u) \ge \tau \ge \frac{\tau}{\epsilon_2}|u|$ if $|u|\in(0,\epsilon_2)$.
	The claim follows as above by  Lemma \ref{lem:ass_q0}.
\end{proof}

\begin{remark}
	\begin{enumerate}

		\item In general, the constant $u_0$ in Theorem \ref{thm:discontinuity} depends on $s$ and the structure of $g$.
		\item We note the second claim concerning $q_0$ in Theorem \ref{thm:discontinuity} holds for all $s>0$ and does not depend on the first claim due to Assumption \ref{assb:5}.
	One can replace $g(u)\ge0$ by the pre-requisite of Lemma \ref{lem:ass_q0}.

		\item Assumption \ref{ass_B} also allows functions of form $g(u)= \tilde q(u)+\delta_{D}(u)$ with some $\tilde g:\R\to\bar\R$ and  the indicator function $\delta_{D}$ of the set $D\subseteq\R$. This means the analysis includes constrained optimization problems, e.g., standard box constraints of form $$\min\limits_{u\in[a,b]} -qu+\frac{1}{2}u^2+s\tilde g(u), $$ with $a,b\in\R, a<b$.

	\end{enumerate}
\end{remark}

\begin{example}
	The proximal map of \eqref{P:min_h} with $g(u)=|u|_0$ is given by the hard-thresholding operator, defined by $$\prox_{sg}(q) = \begin{cases}
	0\quad &\text{if } |q|\leq \sqrt{2s},\\
	q &\text{else}.
	\end{cases} $$
\end{example}

With the above considerations in mind, let us discuss the minimization problem
\begin{equation}\label{PGscal} \min\limits_{u\in\R}
	g_ku+\frac{L}{2}(u-u_k)^2 + g(u).\end{equation}
This minimization corresponds to the pointwise minimization of the integrand in \eqref{eq101}.

\begin{corollary} \label{lem:prototype}
	Let  $ g_k, u_k\in \R,\enspace L>0$ be given. Then the number $u\in\R$ is a solution to \eqref{PGscal} if and only if
	$$u\in\prox_{L^{-1}g}\left(\frac{Lu_k-g_k}{L}\right).		$$
	If $\frac{1}{L}>s_0$, see Theorem \ref{thm:discontinuity}, then
	all global solutions u satisfy $$u=0\, \text{ or }\, |u|\geq u_0(L^{-1})$$ with some  $u_0(L^{-1})>0$ as in Theorem \ref{thm:discontinuity}.  
\end{corollary}
\begin{proof} Problem \eqref{PGscal} is equivalent to $$\min\limits_{u\in \R} \frac{g_k-Lu_k}{L} u+\frac{1}{2}u^2+\frac{1}{L} g(u)$$ and
therefore of form \eqref{P:min_h}. The claim follows by definition and from Theorem \ref{thm:discontinuity}.
\end{proof}

\section{Analysis of the proximal gradient algorithm}\label{section:PG}
In this section, we will analyze the proximal gradient algorithm.

\begin{algorithm}[Proximal gradient algorithm] \label{alg:PG} Choose $L>0$ and $u_0\in L^2(\Omega)$. Set $k=0$.
	\begin{enumerate}
		\item Compute $u_{k+1}$ as solution of
		\begin{equation}\label{PG:subP} \min\limits_{u\in L^2(\Omega)}  f(u_k)+\nabla f(u_k)(u-u_k)+\frac{L}{2}\|u-u_k\|^2_{L^2(\Omega)}+j(u).\end{equation}
		\item Set $k:=k+1$, repeat.
\end{enumerate} \end{algorithm}

The functional to be minimized in \eqref{PG:subP} can be written as an integral functional.
\color{black}
In this representation the minimization can be carried out pointwise by using the previous results.
The following statements are generalizations of \cite[Lemmas 3.10, 3.11, Theorem 3.12]{DW:1},
and the corresponding proofs can be carried over easily.

\begin{lemma} \label{lem:infdimsetting}Let $u_k\in U_{ad}$ be given. Then
	\begin{equation} \label{subP2}
	\min\limits_{u\in L^2(\Omega)} f(u_k)+\nabla f(u_k)(u-u_k) + \frac{L}{2}\|u-u_k\|^2_{L^2(\Omega)} +\int_\Omega g(u(x)) \dx
	\end{equation}
	is solvable, and $u_{k+1}\in L^2(\Omega)$ is a global solution if and only if
	\begin{equation}\label{sol-cond}
	u_{k+1}(x)\in \prox_{L^{-1}g}\left(\frac{1}{L}(Lu_k(x)-\nabla f(u_k)(x))\right)\enspace f.a.a. \enspace x\in\Omega.
	\end{equation}
\end{lemma}
\begin{proof}

Let us show, that we can choose a measurable function satisfying the inclusion \eqref{sol-cond}.
The set-valued mapping $\prox_{L^{-1}g}$ has closed graph and is thus outer semicontinuous. Then by \cite[Corollary 14.14]{var_ana:rockafellar},
the set-valued mapping $x\mapsto \prox_{L^{-1}g}\left(\frac{1}{L}(Lu_k(x)-\nabla f(u_k)(x))\right)$ is measurable.
A well-known result \cite[Corollary 14.6]{var_ana:rockafellar} implies the existence of a measurable function $u$ such that $u(x)\in \prox_{L^{-1}g}\left(\frac{1}{L}(Lu_k(x)-\nabla f(u_k)(x))\right)$
for almost all $x\in \Omega$.
Due to the growth condition of Lemma \ref{lem_prox_growth}, we have $u\in L^2(\Omega)$, and hence $u$ solves  \eqref{subP2}.
If $u_{k+1}$ solves \eqref{subP2} then \eqref{sol-cond} follows by a standard argument.
\end{proof}

We introduce the following notation. For a sequence $(u_k)\subset L^2(\Omega)$ define
$$I_k :=\{x\in\Omega:\enspace u_k(x)\not=0\}, \enspace\chi_k:=\chi(u_k)=\chi_{I_k}.$$
Let us now investigate convergence properties of Algorithm \ref{alg:PG}. The following Lemma will be helpful for what follows.

\begin{lemma}\label{lem:consec} Assume $\frac{1}{L}>s_0$ with $s_0$ from Theorem \ref{thm:discontinuity}. Let $u_k, u_{k+1}\in L^2(\Omega)$ be  consecutive iterates of Algorithm \textnormal{(\ref{alg:PG}).} Then $$\|u_{k+1}-u_k\|^p_{L^p(\Omega)}\geq u_0^p\|\chi_k-\chi_{k+1}\|_{L^1(\Omega)}$$
	holds for $p\in[1,\infty)$,  where $u_0:=u_0(L^{-1})$ is as in Theorem \ref{thm:discontinuity}.
\end{lemma}

 \begin{proof}
 	Since $ u_k(x) \not= 0$ and $u_{k+1}(x)=0$ on  $I_k\setminus I_{k+1}$, it holds $|u_{k+1}(x)-u_k(x)|\geq u_0$ for all $x\in I_k\setminus I_{k+1}$ by Corollary (\ref{lem:prototype}).
 	 Hence,
 	\begin{multline*}
 	\|u_{k+1}-u_k\|^p_{L^p(\Omega)} = \int_\Omega|u_{k+1}(x)-u_k(x)|^p \dx\\
 	\geq\int_{(I_k\setminus I_{k+1})\cup (I_{k+1}\setminus I_k)}|u_{k+1}(x)-u_k(x)|^p \dx
 	\geq u_0^p\|\chi_{k+1}-\chi_k\|_{L^1(\Omega)}.
 	\end{multline*}
\end{proof}

\begin{theorem}\label{thm:conv} For $L>L_f$ let $(u_k)$ be a sequence of iterates generated by Algorithm \ref{alg:PG}. Then the following statements hold:
	\begin{enumerate}[label=(\roman*)]
		\item The sequence $(f(u_k)+j(u_k))$ is monotonically decreasing and converging.
		\item  The sequences $(u_k)$ and $(\nabla f(u_k))$ are bounded in $L^2(\Omega)$ if $f+j$ is weakly coercive on $L^2(\Omega)$, i.e., $f(u)+j(u)\to\infty$ as
		$\|u_k\|_{L^2(\Omega)}\to\infty$.
		\item $\|u_{k+1}-u_k\|_{L^2(\Omega)} \to 0$.
		\item Let $s_0$ be as in Theorem \ref{thm:discontinuity}.
		Assume $\frac{1}{L} > s_0$.
		Then the sequence of characteristic functions $(\chi_k)$ is converging in $L^1(\Omega)$  and pointwise a.e.\@  to some characteristic function $\chi$.
	\end{enumerate}
\end{theorem}
\begin{proof}
\textit{(i)} Due to the Lipschitz continuity of $\nabla f$ it holds $$f(u_{k+1})\leq f(u_k)+\nabla f(u_k)(u_{k+1}-u_k)+\frac{L_f}{2}||u_{k+1}-u_k||^2_{L^2(\Omega)}.$$
Using the optimality of $u_{k+1}$, we find that the inequality
\begin{equation}\label{monotonicity}\begin{split}
f(u_{k+1})+j(u_{k+1})
&\leq  f(u_k)+j(u_k)-\frac{L-L_f}{2}\|u_{k+1}-u_k\|^2_{L^2(\Omega)} \end{split}\end{equation}
holds..
Hence, $(f(u_k)+j(u_k))$ is decreasing. Convergence follows because $f$ and $j$ are bounded from below.

\textit{(ii)}
Weak coercivity of the functional implies that $(u_k)$ is bounded. Furthermore, because of \begin{align*}\|\nabla f(u_k)\|_{L^2(\Omega)}&\leq\|\nabla f(u_k)-\nabla f(0)\|_{L^2(\Omega)}+\|\nabla f(0)\|_{L^2(\Omega)}\\&\leq L_f\|u_k\|_{L^2(\Omega)}+\|\nabla f(0)\|_{L^2(\Omega)},\end{align*} boundedness of $(\nabla f(u_k))$ in $L^2(\Omega)$ follows.

\textit{(iii)} Summation over $k=1,\dots,n$ in \eqref{monotonicity} yields

\begin{equation*}
\sum\limits_{k=1}^n (f(u_{k+1})+j(u_{k+1}))\leq\sum\limits_{k=1}^n \left(  f(u_k)+j(u_k)  -  \frac{L-L_f}{2}\|u_{k+1}-u_k\|^2_{L^2(\Omega)}\right)
\end{equation*} and hence
$$  f(u_{n+1})+j(u_{n+1}) +\sum\limits_{k=1}^n \frac{L-L_f}{2}\|u_{k+1}-u_k\|^2_{L^2(\Omega)}\leq f(u_1)+j(u_1)<\infty.  $$
Letting $n\to\infty$ implies $\sum\limits_{k=1}^\infty\|u_{k+1}-u_k\|^2_{L^2(\Omega)}<\infty$ and therefore $\|u_{k+1}-u_k\|_{L^2(\Omega)}\to 0$.

\textit{(iv)} By Lemma \ref{lem:consec}, we get
\begin{equation*}
\frac{L-L_f}{2}u_0^2\sum\limits_{k=1}^\infty\|\chi_k-\chi_{k+1}\|_{L^1(\Omega)}
\leq \frac{L-L_f}{2}\sum\limits_{k=1}^\infty\|u_k-u_{k+1}\|_{L^2(\Omega)} < +\infty
\end{equation*}
Hence, $(\chi_k)$ is a Cauchy sequence in $L^1(\Omega)$, and therefore also converging in $L^1(\Omega)$, i.e., $\chi_k\to\chi$ for some characteristic function $\chi$.
Pointwise a.e.\@ convergence  of $(\chi_k)$ can be proven by Fatou's Lemma.

\end{proof}

As a consequence, we get the following result.
\begin{corollary} Suppose $\frac{1}{L}>s_0$.
	Then for any weak sequential limit point $u^*\in L^2(\Omega)$ of iterates $(u_k)$ of Algorithm \ref{alg:PG} it holds $$(1-\chi)u^*=0$$ almost everywhere in $\Omega$. Here, $\chi$ is as in Theorem \ref{thm:conv}.
\end{corollary}
\begin{proof}
See \cite[Thm.3.15]{DW:1}.
\end{proof}

\begin{corollary} \label{cor315}
Let $(u_k)$ be a sequence of iterates generated by Algorithm \ref{alg:PG}.
Then $u_{k+1}-u_k\to 0$ pointwise almost everywhere on $\Omega$.
\end{corollary}
\begin{proof}
 By the Lemma of Fatou, we have
 \[
  \int_\Omega \liminf_{n\to\infty} \sum_{k=0}^n |u_{k+1}(x)-u_k(x)|^2 \dx\le  \liminf_{n\to\infty} \sum_{k=0}^n \|u_{k+1}(x)-u_k(x)\|_{L^2(\Omega)}^2<+\infty.
 \]
This implies $\sum_{k=0}^n |u_{k+1}(x)-u_k(x)|^2<\infty$ for almost all $x\in \Omega$, and the claim follows.
\end{proof}

\color{black}

\subsection{Stationarity conditions for weak limit points from inclusions}

Under a weak coercivity assumption Theorem \ref{thm:conv} implies that Algorithm \ref{alg:PG} generates a sequence $(u_k)$ with  weak limit point $u^*\in L^2(\Omega)$. Due to the lack of weak lower semicontinuity in the term $u\mapsto \int_\Omega g(u)\dx$, however, we cannot conclude
anything about the value
of the objective functional in a weak limit point. Unfortunately, we are not able to show $$f(u^*)+j(u^*)\leq \lim\limits_{k\to\infty} f(u_k)+j(u_k),$$ as it was done in {\cite[Thm. 3.14]{DW:1}} for the special choice $g(u):= |u|_0$.
Nevertheless, by using results of set-valued analysis we will show that a weak limit point of a sequence $(u_k)$ of iterates satisfies a certain inclusion in almost every point $x\in\Omega$, which can be interpreted as a pointwise stationary condition for weak limit points.

By definition, the iterates satisfy the inclusion
\[
u_{k+1}(x)\in \prox_{L^{-1}g}\left(\frac{1}{L}(Lu_k(x)-\nabla f(u_k)(x))\right)
\]
for almost all $x\in \Omega$, see e.g., \eqref{sol-cond}.
However, this inclusion seems to be useless for a convergence analysis as the function $u_{k+1}$ to the left of the inclusion
as well as the arguments $Lu_k-\nabla f(u_k)$ only have weakly converging subsequences at best.
The idea is to construct a set-valued mapping $\G:\R\rightrightarrows\R$, such that a solution $u_{k+1}$ of
\eqref{subP2} satisfies the inclusion
\begin{equation}\label{eq38}
u_{k+1}(x) \in \G(z_k(x))
\end{equation}
in almost every point $x\in\Omega$ for some $z_k\in L^2(\Omega)$, where $(z_k)$ converges strongly or pointwise almost everywhere.
Here, we will use
\[
z_k:=-\big(\nabla f(u_k)+L(u_{k+1}-u_k)\big).
\]
By Theorem \ref{thm:conv}, we have $u_{k+1}-u_k\to0$ in $L^2(\Omega)$ and pointwise almost everywhere.
With the additional assumption that subsequences of $(\nabla f(u_k))$ are converging pointwise almost everywhere,
the argument of the set-valued mapping is converging pointwise almost everywhere.
In the context of optimal control problems, such an assumption is not a severe restriction.
So there is a chance to pass to the limit in the inclusion \eqref{eq38}.
\color{black}

\begin{lemma}\label{lem:stat_points} Let $u_{k+1}$ be a solution of \eqref{subP2}. Then
	\begin{equation*}
	u_{k+1}(x) \in \G(z_k(x))\enspace f.a.a.\enspace x\in \Omega,
	\end{equation*}
	where the set-valued mapping $\mathcal G:\R\rightrightarrows\R$ is given by

$$u\in \mathcal G(z):= \G_{L}(z):\Longleftrightarrow u=\argmin\limits_{v\in \R} -zv +\frac{L}{2}(v-u)^2+  g(v).$$
\end{lemma}

Unfortunately, the set-valued map $\G$ is not monotone in general. If $g$ would be convex, then the optimality condition of \eqref{subP2} is $z_k(x) \in \partial g(u_{k+1}(x))$
for almost all $x\in \Omega$, hence one could choose $\G = \gph (\partial (g^*))$, where $g^*$ denotes the convex conjugate of $g$.

\begin{remark}
The definition of $\G$ is related to the concept of $L$-stationary points, introduced in \cite[Definition 9.19]{Beck2014} for $l^0$-optimization problems in $\R^n$.
\end{remark}

For the rest of this section, we will always suppose that $g$ satisfies Assumption \ref{ass_B}.
As a first direct consequence from the definition  of $\G$ we get

\begin{corollary} \label{cor:G_properties}
Assume $\frac{1}{L} > s_0$.
 Let $u,z\in \mathbb R$ such that $u \in \G (z)$. Then we have:
If $u>0$ then $u \ge \max\left(u_0, \frac{Lq_0-z}{L}\right)$, and if $u<0$ then $u \le \min\left(-u_0, -\frac{Lq_0+z}{L}\right)$.
  In case $u=0$ it holds $|z|\le Lq_0$.
Here, $u_0:=u_0(L^{-1})$ and $q_0:=q_0(L^{-1})$ are the positive constants from Theorem \ref{thm:discontinuity}.
\end{corollary}
\begin{proof}
By construction of $\G$, we have
\[
 u\in \G(z) \Longleftrightarrow u = \prox_{L^{-1}g}\left( \frac{Lu+z}{L}\right).
\]
If $u\ne 0$ then by Lemma \ref{lem:prox_sign} and Theorem \ref{thm:discontinuity}, it follows that $u\ge u_0(L^{-1})$  if and only if  $\frac{Lu+z}{L}\geq q_0(L^{-1})$ and likewise  $u<-u_0 (L^{-1})$ iff  $\frac{Lu+z}{L}\leq -q_0(L^{-1})$.  The claim follows for $u>0$ and $u<0$, respectively. On the other hand $u = 0$ is a solution if and only if $|\frac{z}{L}|\leq q_0$, which implies the claim for $u=0$.
\end{proof}

\subsection{A convergence result for inclusions}

Let us recall a few helpful notions and results from set-valued analysis that can be found in the literature, see e.g., \cite{set-val:Aubin,var_ana:rockafellar}.
\begin{definition}
	For a sequence of sets $A_n\subset\R^n$ we define the \textit{outer limit} by
	$$\limsup\limits_{n\to\infty}A_n := \{x:\enspace\exists (x_{n_k}),x_{n_k}\to x,\enspace x_{n_k}\in A_{n_k}\}.$$
\end{definition}

\begin{definition} \label{ch4:defset} Let $S:\R^m\rightrightarrows\R^n$ be a set-valued map.
	\begin{enumerate}
		\item The domain and graph of $S$ are defined by $$\dom S:=\{x:S(x)\not=\emptyset\},\enspace \gph S:=\{(x,y): y\in S(x)\}.$$
		\item $S$ is called \textit{outer semicontinuous} in $\bar x$ if $$\limsup\limits_{x\to\bar x}S(x)\subseteq S(\bar x).$$
		\item $S$ is called \textit{locally bounded} at $x\in\R^m$ if  there is a neighborhood $U$ of $x$ such that $S(U)$ is bounded.
\end{enumerate}\end{definition}

A set-valued mapping $S$ is outer semicontinuous if and only if it has a closed graph.

The following convergence analysis relies on \cite[Thm. 7.2.1]{set-val:Aubin}.
We want to extend this result to set-valued maps into $\R^n$ that are not locally bounded.
Let us define the following set-valued map that serves as a generalization of $x \to \conv(F(x))$
for the locally unbounded situation.

\begin{definition}
Let $F:\R^m\rightrightarrows \R^n$ be a set-valued map.

Define the set-valued map $\conv^\infty F:\R^m\rightrightarrows \R^n$ by
\[
 (\conv^\infty F)(x) := \limsup_{k\to \infty} \conv  \left(F\left( x + B_{1/k}(0) \right) \right).
\]

\end{definition}

By definition, it holds $\gph F \subset \gph \conv^\infty F$.
In addition, we have $ \overline{\conv}(F(x)) \subset  (\conv^\infty F)(x)$.
If $F$ is locally bounded in $x$, then $(\conv^\infty) F(x) =  \overline{\conv}(F(x))$,
which can be proven using Carathéodory's theorem.
In general, $\dom \conv^\infty F$ is strictly larger than $\dom F$.

\begin{example}

 Define $F:\R \rightrightarrows \R$ by
 \[
  \gph F = \{(x,y):\  yx=1\}.
 \]
 Then $F$ is not locally bounded near $x=0$.
 Here it holds  $\gph( \conv^\infty F) = \gph F \cup ( \{0\} \times \R)$, so that $\dom ( \conv^\infty F) = \R \ne \dom F$.

\end{example}

\begin{theorem}\label{thm:conv_setval}
	Let $(\Omega,\mathcal{A},\mu)$ be a measure space and
		$F:\R^m\rightrightarrows \R^n$ be a	set-valued map.
	Let sequences of measurable functions $(x_n),(y_n) $  be given such that
	\begin{enumerate}
		\item $x_n $ converges almost everywhere to some function $x:\Omega\to \R^m$,
		\item $y_n$ converges weakly to a function $y$ in $L^1(\mu,\R^n)$,
		\item $y_n(t)\in F(x_n(t))$ for almost all $t\in\Omega$.
	\end{enumerate}
Then for almost all $t\in\Omega$ it holds:
\[
y(t)\in (\conv^\infty F)(x(t)).
\]
\end{theorem}
\begin{proof}

Arguing as in the proof of  \cite[Thm. 7.2.1]{set-val:Aubin}, we find
\[
 y(t) \in \bigcap_{k\in \mathbb N} \overline{\conv} \left( F( x(t) + B_{1/k}(0)) \right)
\]
for almost all $t\in \Omega$. Take $t\in \Omega$ such that the above inclusion is satisfied. Then there is a sequence $(u_k)$
such that $u_k \to y(t)$, $u_k \in \conv (F( x(t) + B_{1/k}(0)))$.
This implies
$y(t) \in \limsup_{k\to\infty}  \conv\left( F( x(t) + B_{1/k}(0)) \right)$,
or equivalently $y(t)\in (\conv^\infty F)(x(t))$.
\end{proof}

\subsection{Stationarity conditions for weak limit points}

Recall, for iterates $(u_k)$  of Algorithm \ref{alg:PG} and the corresponding sequence $z_k$ we have by construction
\begin{equation*}
u_{k+1}(x)\in{\G}(z_k(x))\enspace\text{ f.a.a. }\enspace x\in \Omega.
\end{equation*}

Then by Theorem \ref{thm:conv_setval}, we could expect the inclusion
$u^*(t)\in (\conv^\infty \G)(-\nabla f(u^*)(x))$ pointwise almost everywhere to hold in the subsequential limit.
However, the convexification of $\G$ results in a set-valued map that is very large.
In order to obtain a smaller inclusion in the limit, we will employ the result of Corollary \ref{cor:G_properties}:
the graph of $\G$ can be split into three clearly separated components.
In the sequel, we will show that we can pass to the limit with each component separately, which leads to a smaller set-valued map in the limit.
\color{black}
This observation motivates the following splitting of the map $\G$.

\begin{definition} \label{def:mapsG}
For $L>0$ we define the following set-valued mappings.
	\begin{enumerate}
		\item $\G^+:\R\rightrightarrows\R$ with $u\in \G^+(z)$ $:\Longleftrightarrow$ $u\in \G(z)$ and $u>0$,
		\item $\G^-:\R\rightrightarrows\R$ with $u\in \G^-(z)$ $:\Longleftrightarrow$ $u\in \G(z)$ and $u<0$,
		\item $\G^0:\R\rightrightarrows\R$ with $u\in \G^0(z)$  $:\Longleftrightarrow$ $u\in \G(z)$ and $u=0$.
	\end{enumerate}
\end{definition}

The mappings $\G^+,\G^-$ and $\G^0$ are depicted in Figure \ref{fig:mapG} for the special choice $g(u):=\frac\alpha2|u|^2+|u|^p+\delta_{[-b,b]}(u),\enspace p\in(0,1),b\in(0,\infty)$.

Obviously we have by construction
\begin{equation}\label{incl_all}
u_{k+1}(x)\in (\G^+\cup\G^-\cup\G^0)(z_k(x)) \quad \text{ f.a.a. } x\in \Omega.
\end{equation}

\begin{corollary}
	The mappings  $\G,\G^0$ are outer semicontinuous.  If  $L^{-1}>s_0$ the same holds for $\G^+ $ and $ \G^-$.
\end{corollary}
\begin{proof}
	$\G$ being outer semicontinuous is equivalent to the closedness of its graph. Let $(u_n),(q_n)$ be sequences such that $u_n\to u, q_n\to q$ and $u_n\in \G(q_n)$.
	By definition it holds
	\begin{align*}
	0\leq -q_n(v-u_n)+(g(v)-g(u_n))+\frac{L}{2}(v-u_n)^2
	\end{align*} for all $v\in\R$. Passing to the limit in above inequality yields
	\begin{align*}
	0\leq -q(v-u)+(g(v)-g(u))+\frac{L}{2}(v-u)^2
	\end{align*}
 due to the lower semicontinuity of $g$. Hence,
	$$u=\argmin\limits_{v\in\R} -qv+\frac{L}{2}(v-u)^2+ g(v),$$ i.e., $u\in\G(q)$, which is the claim for $\G.$
	For $\G^+,\G^-,\G^0$ the claim follows as their graphs are intersections of closed sets with $\gph\G$,
	which follows from Corollary \ref{cor:G_properties} (for suitable chosen $L$ in case of $\G^+,\G^- ).$
\end{proof}

In the sequel we want to  apply Theorem \ref{thm:conv_setval} to
each of the set-valued maps in \eqref{incl_all} separately.
Let us first show the next helpful result.

\begin{lemma}\label{lem324}
	Let $(u_k)$ be a sequence of iterates generated by Algorithm \ref{alg:PG}. Let  $b>a$ be given. Define
	\[
	A_k^+ := \{ x\in \Omega:\ u_k(x) \ge b\},
	\]
	\[
	A_k^- := \{ x\in \Omega:\ u_k(x) \le a\},
	\]
	and $\chi_k^+:=\chi_{A_k^+}$, $\chi_k^-:=\chi_{A_k^-}$.
	Then it holds
	\[
	\sum_{k=1}^\infty \|\chi_{k+1}^+\chi_k^- + \chi_{k+1}^-\chi_k^+\|_{L^1(\Omega)}<+\infty.
	\]

	If $\chi_k^++\chi_k^-=1$ for all $k$ almost everywhere, then there are characteristic functions
	$\chi^+,\chi^-$ such that $\chi^++\chi^-=1$ almost everywhere,
	$\chi_k^+\to \chi^+$ and $\chi_k^-\to \chi^-$ strongly in $L^1(\Omega)$ and pointwise almost everywhere.

\end{lemma}
\begin{proof}
	Let $x\in\Omega$.
	If $\chi_{k+1}^+(x)\chi_k^-(x)=1$, then $u_{k+1}(x) - u_k(x) \ge b-a$.
	This proves $\|\chi_{k+1}^+\chi_k^-\|_{L^1(\Omega)} \le (b-a)^{-2} \|u_{k+1}-u_k\|_{L^2(\Omega)}^2$.
	Similarly, we obtain $\|\chi_{k+1}^-\chi_k^+\|_{L^1(\Omega)} \le (b-a)^{-2} \|u_{k+1}-u_k\|_{L^2(\Omega)}^2$.
	Since $\sum_{k=1}^\infty \|u_{k+1}-u_k\|_{L^2(\Omega)}^2<+\infty$, the claim follows.
	Suppose $\chi_k^++\chi_k^-=1$ almost everywhere. Then we have
	\[
	 \chi_{k+1}^+\chi_k^- + \chi_{k+1}^-\chi_k^+ = \chi_{k+1}^+(1-\chi_k^+) + (1-\chi_{k+1}^+)\chi_k^+ = |\chi_{k+1}^+-\chi_k^+|,
	\]
	which implies the second claim.
\end{proof}

\begin{theorem}\label{thm:set_valconv}
Let $s_0$ be as in Theorem \ref{thm:discontinuity}.
		Assume $\frac{1}{L} > s_0$.
	Let $(u_k)$ be a sequence of iterates generated by Algorithm \ref{alg:PG} with weak limit point $u^*\in L^2(\Omega)$, i.e., $u_{k_n}\rightharpoonup u^*$.
	Assume $\nabla f(u_{k_n})(x) \to \nabla f(u^*)(x)$ for almost every $x\in \Omega$.
	Let ${\mathcal G}^0, {\mathcal G}^+,\G^-:\R\rightrightarrows\R$ be as in Definition \ref{def:mapsG}.  Then
	\[
	  u^*(x)\in \left(\G_0\cup {\conv}^\infty \G^+\cup\conv^\infty\G^-\right)(-\nabla f(u^*)(x))
	  \]
	holds for almost all $x\in\Omega$.
\end{theorem}

\begin{proof}

	By Theorem \ref{thm:conv} and Corollary \ref{cor315},
	we have $u_{k_n+1}\rightharpoonup u^*$ in $L^2(\Omega)$ and
	\[
	z_{k_n}:=-\left(\nabla f(u_{k_n})+L(u_{k_n+1}-u_{k_n})\right)\to-\nabla f(u^*):=z
	\]
	pointwise almost everywhere on $\Omega$.
	Let us define $I^+_k:=\{x\in\Omega:\:u_k(x)>0\}$ and  $I^-_k:=\{x\in\Omega:\:u_k(x)<0\}$ with associated characteristic functions $\chi_k^+,\chi_k^-$.
	Then by Lemma \ref{lem324} with $a=0$ and $b=u_0$ with $u_0$ from Theorem \ref{thm:discontinuity}, we obtain
	$\chi_k^+\to \chi^+$ in $L^1(\Omega)$ and pointwise almost everywhere.
	Similarly, $\chi_k^-\to \chi^-$ in $L^1(\Omega)$ and pointwise almost everywhere.

	Let us fix $(u',q')\in \gph\G^+$. Then the following inclusion
	\[
	\chi_{k+1}^+u_{k+1} + (1-\chi_{k+1}^+)u' \in \G^+(\chi_{k+1}^+z_k + (1-\chi_{k+1}^+)q')
	\]
	is satisfied almost everywhere on $\Omega$. By Theorem \ref{thm:conv_setval}, we obtain
	\[
	\chi^+ u^* + (1-\chi^+)u' \in {\conv}^\infty\G^+(\chi^+z + (1-\chi^+)q')
	\]
	almost everywhere on $\Omega$. Similarly, we obtain for $(u'',q'')\in \gph G^-$
	\[
	\chi^- u^* + (1-\chi^-)u'' \in{\conv}^\infty \G^-(\chi^-z + (1-\chi^-)q'')
	\]
	and
	\[
	 (1-\chi) u^*  \in \G^0 ( (1-\chi) z )
	\]
	almost everywhere, where $\chi_k$ and $\chi$ are as in Theorem \ref{thm:conv}.
	Note that ${\conv}^\infty \G^0=\G^0$.
	By construction, $\chi_k^+ + \chi_k^- = \chi_k$, which implies $\chi^++\chi^- = \chi$.
	Then we can combine all the inclusions above into one, which is
	\[
	 u^+ (x) \in \left(\G_0\cup {\conv}^\infty \G^+\cup\conv^\infty\G^-\right)(-\nabla f(u^*)(x))
	\]
	for almost all $x\in \Omega$.
\color{black}
\end{proof}

Let us remark that the assumption of pointwise convergence of $(\nabla f(u_k))$ is not a severe restriction.
If $\nabla f:L^2(\Omega) \to L^2(\Omega)$ is  completely continuous, then this assumption is fulfilled.
For many control problems, this property of $\nabla f$ is guaranteed to hold.

Interestingly, we can get rid of the convexification operator $\conv^\infty$ if we assume that the whole sequence
$(\nabla f(u_k))$ converges pointwise almost everywhere.

\begin{theorem}\label{theo420}
Let $(u_k)$ be a sequence of iterates generated by Algorithm \ref{alg:PG} with weak limit point $u^*\in L^2(\Omega)$.
Assume $\nabla f(u_k) \to \nabla f(u^*)$ pointwise almost everywhere.
Then
\[
  u^*(x) \in \G( -\nabla f(u^*)(x))
\]
holds for almost all $x\in\Omega$.
\end{theorem}
\begin{proof}
Denote $z(x):= -\nabla f(u^*)(x)$. Then $z_k(x)\to z(x)$ almost everywhere.

Let $(\tilde z, \tilde u) \not\in \gph \G$. Since $\gph \G$ is closed,
there is $\epsilon>0$ such that
\[
 \left( B_\epsilon(\tilde z) \times  B_\epsilon(\tilde u) \right) \cap \gph \G = \emptyset.
\]
  Let $\epsilon'\in (0,\epsilon)$.
  Set
  \[
  I:= \{ x: \ |\tilde z- z(x)|<\epsilon'\},
  \]
  and
  \[
  I_K:=\{x\in I:\ |\tilde z-z_k(x)| < \epsilon \quad \forall k>K\}.
  \]
  The sequence $(I_K)$ is monotonically increasing.
  Since $z_k(x)\to z(x)$ for almost all $x\in\Omega$, we have $\cup_{K\in \N} I_K = I$.

  Define
  \[
  A_k^+ := \{ x\in \Omega:\ u_k(x) \ge \tilde u + \epsilon\},
  \]
  \[
  A_k^- := \{ x\in \Omega:\ u_k(x) \le \tilde u - \epsilon\},
  \]
  and $\chi_k^+:=\chi_{A_k^+}$, $\chi_k^-:=\chi_{A_k^-}$.
  By Lemma \ref{lem324} above, we have
  $\sum_{k=1}^\infty \|\chi_{k+1}^+\chi_k^- + \chi_{k+1}^-\chi_k^+\|_{L^1(\Omega)}<+\infty$,
  $\chi_{k+1}^+\chi_k^- + \chi_{k+1}^-\chi_k^+\to 0$ in $L^1(\Omega)$ and
  pointwise almost everywhere.

  Let $x\in I$. Then there is $K$ such that $x\in I_K$.
  This implies $u_{k}(x) \not \in B_\epsilon(\tilde u)$ for all $k>K$.
  Here, the pointwise convergence of the whole sequence $(z_k)$ is needed.
  The sum $\sum_{k=K+1}^\infty (\chi_{k+1}^+\chi_k^- + \chi_{k+1}^-\chi_k^+)(x)$ counts the number of switches
  between values larger than $\tilde u+\epsilon$ and smaller than $\tilde u-\epsilon$ from $u_k(x)$ to $u_{k+1}(x)$.
  Since this sum is finite for almost all $x\in \Omega$, there is only a finite number of such switches.
  Then there
  is $K'>K$ such that either $u_k(x)\ge \tilde u+\epsilon$ for all $k>K'$ or
  $u_k(x)\le \tilde u-\epsilon$ for all $k>K'$.
  Set
  \[
  S_K^+:= \{ x\in I: \ u_k(t) \ge \tilde u+\epsilon  \quad \forall k>K\},
  \]
  \[
  S_K^-:= \{ x \in I: \ u_k(t) \le \tilde u-\epsilon \quad \forall k>K\}.
  \]
  The sequences $(S_K^+)$ and $(S_K^-)$ are increasing, and $\cup_{K\in \N} (S_K^+\cup S_K^-) = I$.

  Since $u_{k_n}\rightharpoonup u^*$, this implies $u^*\ge \tilde u+\epsilon$ on $S_K^+$ and
  $u^*\le \tilde u-\epsilon$ on $S_K^-$.
  Since $\cup_{K\in \N} (S_K^+\cup S_K^-) = I$,
  this implies
  \[
   u^*(x) \not \in B_\epsilon(\tilde u)
  \]
  for almost all $x\in I$, which implies
  \[
   ((z(x), u^*(x))  \not \in B_{\epsilon'}(\tilde z)\times B_\epsilon(\tilde u)
   \]
   for almost all $x\in \Omega$. Since we can cover the complement of $\gph \G$ by countably many such sets,
   the claim follows.
\end{proof}

For convex functions $g$, the result above is equivalent to
\[
 -\nabla f(u^* ) \in \partial g(u^*),
\]
see, e.g., \cite[Cor. 27.9]{bauschkecombettes}.

\subsection{Pointwise convergence of iterates}
So far we were able to show that weak limit points  of iterates $(u_k)$ satisfy a certain inclusion  in a pointwise sense. However, the resulting set in the limit might still be large or  even  unbounded in general.
	Assuming that $\G$ is (locally) single-valued on its components $\G^+,\G^-,\G^0$, we can show local pointwise convergence of a subsequence of iterates $(u_{k_n})$ to a weak limit point $u^*\in L^2(\Omega)$.
	In the next result this is illustrated for the map $\G^+$, however it can be shown for the components $\G^-,\G^0$ similarly.
To this end, we set in the following  $\chi_k^+:=\chi_{\{x\in\Omega: \:u_k(x)>0\}}$  with   $\chi_k^+\to\chi^+$ in $L^1(\Omega)$ and pointwise almost everywhere  by Lemma \ref{lem324}.

\begin{theorem}
\label{theo421}
Let $\bar z\in \dom(\G^+)$. Assume that $\G^+:\R\to\R$ is single-valued and locally bounded on $B_\epsilon(\bar z)\cap\dom(\G^+)$ for some $\epsilon>0$.
Let $u_{k_n}\rightharpoonup u^*$ in $L^2(\Omega)$  and assume $\nabla f(u_{{k_n}})(x)\to \nabla f(u^*)(x)$ pointwise almost everywhere.
For $\epsilon' \in (0,\epsilon]$ define the set
\[
I_{{\epsilon'}}:=\left\{x\in\supp(\chi^+): -\nabla f(u^*)(x)\in B_{\epsilon'}(z)\cap\dom(\G^+)\right\}.
\]
Then $$u_{k_n}(x)\to u^*(x)$$ holds for almost all $x\in I$. Furthermore, we have
\[
u^*(x)\in \prox_{L^{-1}g}\left(\frac{1}{L}(Lu^*(x)-\nabla f(u^*)(x))\right)\text{ f.a.a. } x\in I_{\epsilon}.
\]
\end{theorem}

\begin{proof}
	Let $u_{k_n+1}\rightharpoonup u^*$ in $L^2(\Omega)$. By the assumption and
	Corollary \ref{cor:G_properties}
	it holds $z_{k_n}(x)\to z(x):=-\nabla f(u^*)(x)$ pointwise almost everywhere. In addition, $u_{k_n+1}\rightharpoonup u^*$ in $L^2(\Omega)$ holds.
	Let $\epsilon'\in (0,\epsilon)$ be given. Take
	$x\in I_{\epsilon'}$
	such that  $z_{k_n}(x)\to z(x)$. Then there is $K>0$ such that $|z_{k_n}(x)-\bar z|<\epsilon$ for all $k_n>K$. Since $x\in\supp(\chi^+)$ and $\chi_k^+\to\chi^+$ in $L^1(\Omega)$ and pointwise almost everywhere
	there is $K'>0 $ such that  $x\in \supp(\chi^+_k)$ for all $k>K'$.
	Hence, for $k_n$ sufficiently large we have
	\[
	z_{k_n}(x)\in B_\epsilon(\bar z)\cap \dom(\G^+).
	\]
	Since $\G^+$ is single-valued, locally bounded and outer semicontinuous in $B_\epsilon(\bar z)\cap \dom(\G^+)$, it is continuous, see also \cite[Cor. 5.20]{var_ana:rockafellar}.
	This implies
	\[
	\lim\limits_{n\to\infty}u_{k_n+1}(x) =\lim\limits_{n\to\infty} \G^+(z_{k_n}(x)) =\G^+(\lim\limits_{n\to\infty}z_{k_n}(x) )= \G^+(z(x)).
	\]
	The continuity property mentioned above implies $\conv^\infty\G^+(z(x)) = \G^+(z(x))$.
	Then by Theorem  \ref{thm:set_valconv}, $\G^+(z(x)) = \{u^*(x)\}$, and the convergence $u_{k_n}(x)\to u^*(x)$ follows.
	The fixed-point property is a consequence of the closedness of the graph of the proximal operator. As $x\in I_{\epsilon'}$ was chosen arbitrary,
	and $I_\epsilon = \cup_{\epsilon'\in(0,\epsilon)} I_{\epsilon'}$, the claim is proven.
\end{proof}

	The above result requires local boundedness of the set-valued map $\G$, which is not satisfied  in general.
	For some interesting choices of $g$, e.g. $g(u):=|u|^p$, it can be proven, see Section \ref{sec5}.
Let us give an example of a locally unbounded map $\G$ below.

\begin{example} Let $L>0$ and define $g(u):=\delta_{\mathbb{Z}}(u):=\begin{cases}0\quad&\text{ if }
u\in\mathbb{Z}\\+\infty&\text{else}.\end{cases}$ with the associated map $\G_L$. Set  $U:=[-\frac L2,\frac L2]$.
Then it holds that $\G(z)=\mathbb Z$  for all $z\in U$, i.e., $\G$ is clearly not locally bounded in the origin.
\end{example}

\subsection{Strong convergence of iterates}
	Many optimal control problems of type \eqref{P} include a smooth cost functional of form $u\to\frac\alpha2\|u\|^2_{L^2(\Omega)},\:\alpha>0$.
	For the rest of the sequel, we will treat this term explicitly in the convergence analysis to obtain an almost everywhere and strong convergence of a subsequence. Therefore let  $\tilde g:\R\to\R$ satisfy Assumption \ref{ass_B} and  consider a sequence of iterates computed by
	\begin{equation}\label{P_alpha}
	u_{k+1}:=\argmin\limits_{u\in  L^2(\Omega)}f(u_k)+\nabla f(u_k)(u-u_k)+\frac{L}{2}\|u-u_k\|^2_{L^2(\Omega)}+\frac\alpha2\|u\|^2_{L^2(\Omega)}+\int_\Omega\tilde g(u(x))\dx.
	\end{equation}
    The  solution to \eqref{P_alpha} is now given by
    \[
	u_{k+1}(x)\in \prox_{\frac{1}{L+\alpha}\tilde g}\left(\frac{1}{L+\alpha}(Lu_k(x)-\nabla f(u_k)(x))\right)
	\]
	for almost every $x\in\Omega$.
	It follows that all the analysis that was done in this section still applies in this case and all results can be transferred except for a possible change of notation.
	Furthermore, we adapt the set-valued map $\G:\R\to\R$ from Lemma \ref{lem:stat_points} which is  then defined by
	$$u\in \mathcal G(z):\Longleftrightarrow u=\argmin\limits_{v\in \R} -zv +\frac{L}{2}(v-u)^2+ \frac\alpha2v^2+ \tilde g(v).$$

 For simplicity we assume $\text{dom}(\tilde g)=[-b,b]$ with $b\in(0,\infty]$, i.e., the subproblem \eqref{P_alpha} is equivalent to a box constrained optimization problem of form
\[
u_{k+1}:=\argmin\limits_{u\in  L^2(\Omega)}f(u_k)+\nabla f(u_k)(u-u_k)+\frac{L}{2}\|u-u_k\|^2_{L^2(\Omega)}+\frac\alpha2\|u\|^2_{L^2(\Omega)}+\int_\Omega\tilde g(u(x))\dx.
\]
subject to $|u(x)|\leq b$ for almost every $x\in\Omega$.
To obtain strong convergence of iterates in $L^1(\Omega)$ and an $L$-stationary condition almost everywhere, we need to put
stronger and more restricting assumptions on $\tilde g$, as the next theorem shows.
To this end, let us introduce the following extension of Assumption \ref{ass_B}.
\paragraph{Assumption B$^+$.}
\begin{enumerate}[label=(B\arabic{enumi})]
	\setcounter{enumi}4
	\item\label{assb:6}  $\tilde g$ is $C^1$ on $(0,b)$ with $g'(b) := \lim_{u\nearrow b}g'(u)$.
	\item\label{assb:7} 
	For $s>0$  there is $u_I:=u_I(s)>0$ such that $u\mapsto \frac12 u^2+s\tilde g(u)$ is strictly convex on $[u_I,b]$.
\end{enumerate}

First, we have the following necessary optimality condition for \eqref{P_alpha} due to Assumption \ref{assb:6}.
\begin{corollary} \label{cor:varinptw}
	Let $u_{k+1}$ be a solution to \eqref{P_alpha} and $\tilde g$ satisfy in addition \ref{assb:6}. Then the pointwise inequality in $\R$
	\begin{align*}\big(\nabla f(u_k)(x)&+L(u_{k+1}(x)-u_k(x))+\alpha u_{k+1}(x)\\&+ \tilde g'(u_{k+1}(x))\big)(v-u_{k+1}(x))\geq 0\end{align*}for all $v\in [-b,b]$ holds for almost all $x\in I_{k+1}$.
\end{corollary}
\begin{proof}
	 Since $\text{dom}(g) = [-b,b]$,
	 minimizing the integrand in 	\begin{equation}\label{inteq}\min\limits_{u\in L^2(\Omega)}\quad\int_\Omega \nabla f(u_k)(x)u(x)+\frac{L}{2}(u(x)-u_k(x))^2+\frac{\alpha}{2}|u(x)|^2+ \tilde g(u(x)) \dx.\end{equation} pointwise is equivalent to solve the constrained problem
	  $$\min\limits_{u:|u|\leq b}f(u_k)(x)u+\frac{L}{2}(u-u_k(x))^2+\frac{\alpha}{2}|u|^2+ \tilde g(u)$$
	  in every Lebesgue point $x$.
	  For $x\in I_{k+1}$  it holds $u_{k+1}(x)\not=0$, and therefore
	above problem is differentiable. The claimed inequality is the corresponding  necessary optimality condition.
\end{proof}

Let us for the rest of the sequel assume that $\tilde g$ satisfies \ref{assb:6} and \ref{assb:7} in addition to Assumption \ref{ass_B}. This enables us to give more information about the set-valued map $\G$ as the next result shows. That is, elements in $\G$ are  (possibly unique) solutions of an associated variational inequality.

\begin{lemma} \label{lem328_}Let $u_0 ({\frac{1}{L+\alpha}}),q_0({\frac{1}{L+\alpha}})$ be constants as in Theorem \ref{thm:discontinuity} and  $|u|\geq u_0({\frac{1}{L+\alpha}})$. Then
	$u\in\G(z)$ satisfies the variational inequality
	\begin{equation} \label{lem328:VI}
	(-z+\alpha u+ \tilde g'(u))(v-u)\geq 0
	\end{equation}
	for all $v\in[-b,b]$.
	If in addition $|\frac{z+Lu}{L+\alpha} |\ge q_0\left(\frac1{L+\alpha}\right)$ and $u_0\geq u_I$ with $u_I:=u_I({\frac{1}{L+\alpha}})$ as in \ref{assb:7}, then we have $u\in\G(z)$ if and only $u$ satisfies \eqref{lem328:VI}.
\end{lemma}
\begin{proof}Let  us discuss the case $u\geq u_0$ only.
	If $u\in\G(z)$ for some $z\in\R$, then by definition
	\begin{align*}
	u&= \argmin\limits_{v\in \R}-zv+\frac{L}{2}(v-u)^2+\frac{\alpha}{2}v^2+\tilde g(v)\\
	&= \argmin\limits_{|v|\leq b}-zv+\frac{L}{2}(v-u)^2+\frac{\alpha}{2}v^2+\tilde g(v)\\
	&=\argmin\limits_{|v|\leq b} -(z+Lu)v+\frac{L+\alpha}{2}v^2+\tilde g(v)
	\end{align*}
	Hence, by first order necessary optimality condition it holds
	\begin{align*}\label{eq:lem4.1}
	0\leq &\left(-(z+Lu)+(L+\alpha)u+ \tilde g'(u)\right)(v-u)\\
	&= (-z+\alpha u+\tilde g'(u))(v-u) \notag
	\end{align*}
	for all $v\in [-b,b]$, which is the claim.

	Assume  $u_I\leq u_0$ holds, and let $u>0$ satisfy \eqref{lem328:VI}, then $u$  satisfies in particular
	$$0\leq(-z+\alpha u+\tilde g'(u))(v-u)=(-z-Lu+(\alpha+L)u+\tilde g'(u))(v-u)$$ for all $v\in[u_I,b]$, i.e., it is stationary to

	\begin{equation} \label{min_Pz}
	\min\limits_{v\in[u_I,b]}-zv+\frac{\alpha}{2}v^2+ \tilde g(v)
	\end{equation}
	and also to
	$$\min\limits_{v\in[u_I,b]}-(z+Lu)v+\frac{L+\alpha}{2}v^2+\tilde g(v).$$
	By convexity  $u$ is the unique solution of the latter and since by assumption $\frac{z+Lu}{L+\alpha} \ge q_0\left(\frac1{L+\alpha}\right)$, it follows  from Theorem \ref{thm:discontinuity} that there is a global solution  larger than $u_0$ to the unconstrained problem which together implies  $u\in\G(z)$.

\end{proof}

\begin{lemma} \label{lem:330} Let $\alpha>0$. Assume $u_{k+1}$ is a global solution to
\eqref{P_alpha} with
	 $|u_{k+1}(x)|\geq u_0\geq u_I(\frac1\alpha)$ for almost all $x\in I_{k+1}$, where $u_I(\frac1\alpha)$ is as in \ref{assb:7}. Then there is a continuous mapping $G:L^2(\Omega)\to L^2(\Omega)$ such that $$u_{k+1} =\chi_{k+1}G\left(\frac{z_k}{\alpha}\right).$$
\end{lemma}
\begin{proof} We set $s:=\frac{1}{\alpha}$ and $u_{I}:=u_I(s)$ as in \ref{assb:7}.
	Note that by assumptions the following holds for $\alpha>0$ and $|u|\geq u_0\geq  u_{I}$:
	$$u\in\G(z)\Longleftrightarrow u\in \prox_{(L+\alpha)^{-1}\tilde g}\left(\frac{z+Lu}{L+\alpha}\right)\Longrightarrow u\in\prox^{u_I}_{s\tilde g}\left(\frac z\alpha\right),$$
	where we define, corresponding to \eqref{min_Pz},
	$$u\in\prox^{u_I}_{s\tilde g}(z) :\Longleftrightarrow u =  \argmin\limits_{|v|\in[u_I,b]}-zv+\frac{1}{2}v^2+s\tilde  g(v).$$

Due to assumption \ref{assb:7} and Lemma \ref{lem328_}, $u_{k+1}(x)$ is the only element in $\G(z_k(x))\setminus\{0\} $ for almost all $x\in I_{k+1}$ and it holds $u_{k+1}(x) =\prox^{u_{I}}_{s\tilde g}\left(\frac{ z_k(x)}{\alpha}\right)$. Set
\[
z_I:=\sup\{q>0: u_I=\prox_{s\tilde g}^{u_I}(q)\}.
\]
It is easy to see that $\prox_{s\tilde g}^{u_{I}}$ is single-valued for $|z|>0$.  Since it is in addition outer semicontinuous and locally bounded for $|z|\ge z_I$, it is also continuous on $\{z:\:|z|\ge z_I\},$ see also \cite[Corollary 5.20]{var_ana:rockafellar}.
 Let $u\in \prox_{s\tilde g}^{u_I}(z)$.
By optimality of $u$ we have
\[
-zu+\frac12u^2+s\tilde g(u)\leq -z\cdot
{ \sign(u)u_I}+\frac12 u_I^2+s\tilde g(u_I).
\]
Dividing by $|u|>0$, we get
\[
\frac12|u|\leq\left(\frac{u-\sign(u)u_I}{|u|}\right)z+\frac{u_I^2}{|u|}+s\frac{\tilde g(u_I)-\tilde g(u)}{|u|}.
\]
Having in mind that $\frac{u_I}{|u|}\le1$, the growth estimate  $|\prox_{s\tilde g}^{u_I}(z)|\leq 2|z|+c$ for all $|z|\ge z_I$ with some $c>0$ independent of $z$ follows.

Let $l:\R\to\R$ denote a continuous function defined by
\[ l(z) := \begin{cases}
\prox_{sg}^{u_I}(z) &\text{ if } |z|\geq z_I,\\
\frac{u_I}{z_I}z &\text{ if } |z|\le z_I.
\end{cases}
\]
Define
$$G:L^2(\Omega)\to L^2(\Omega),\: G(z)(x)=l(z(x))$$ for $z:\Omega\to\R$.  
Then by a well-known result, see e.g. \cite[Theorem 3.1]{Appell1990},
the superposition operator $G$ is continuous from $L^2(\Omega)\to L^2(\Omega)$
and the claim follows.
\end{proof}

Now, we are able to prove strong convergence of a subsequence of $(u_k)$ similar to {\cite[Thm. 3.17]{DW:1}}.
\begin{theorem}\label{thm:strongconv}
Suppose complete continuity of $\nabla f$
 and let $(u_k)\subset L^2(\Omega)$ be a sequence generated by Algorithm \ref{P_alpha} with weak limit point $u^*$. Under the same assumptions as in Lemma \ref{lem:330} $u^*$ is a strong sequential limit point of $(u_k)$ in $L^1(\Omega)$.
\end{theorem}

\begin{proof} By Lemma \ref{lem:330} there exists a continuous mapping $G:L^2(\Omega)\to L^2(\Omega)$ such that $u_{k+1}=\chi_{k+1}\left(G(\frac{z_k}{\alpha})\right)$.
	Let $u_{k_n}\rightharpoonup u^*$ in $L^2(\Omega)$.
	 Again, by Theorem \ref{thm:conv} and complete continuity of $\nabla f$, we obtain strong convergence of the sequence $$z_{k_n}  := -\left(\nabla f(u_{k_n})+L(u_{k_n+1}-u_k) \right)\to  -\nabla f(u^*) =:z^* $$ in $L^2(\Omega)$ as well as $\chi_k\to \chi$ in $L^p(\Omega)$ for all $p<\infty$ and $u_{k_n+1}\rightharpoonup u^*$.
	Then the convergence
	\[
	u_{k_n+1}=\chi_{k_n+1}G\left(\frac{1}{\alpha}z_{k_n}\right)\to \chi G\left(\frac{1}{\alpha}z^*\right)
	\]
	in $L^1(\Omega)$ follows by Hölder's inequality. Since strong and weak limit points coincide, it follows $u_{k_n}\to u^*$ in $L^1(\Omega)$ and
	$$u^* =\chi G\left(-\frac{1}{\alpha}\nabla f(u^*) \right).$$
\end{proof}

With the assumptions in Theorem \ref{thm:strongconv}  we can find an almost everywhere converging subsequence of iterates, i.e., $u_{k_n}(x)\to u^*(x)$ for almost every $x\in\Omega$. By the closedness of the mapping $\prox_{s\tilde g}$, we get
\begin{equation}\label{fixpoint}u^*(x) \in \prox_{\frac{1}{L+\alpha}\tilde g}\left(\frac{1}{L+\alpha}(Lu^*(x)-\nabla f(u^*)(x))\right)\enspace \text{f.a.a } x\in\Omega, \end{equation} i.e., $u^*$ is  $L$-stationary to the problem in almost every point.
If $L=0$ in \eqref{fixpoint}, then we obtain by Lemma \ref{lem:infdimsetting} $$u^*(x) = \argmin\limits_{u\in\R}f(u_k)(x)u(x)+\frac{\alpha}{2}|u(x)|^2+\tilde  g(u(x)) \enspace\text{ f.a.a. } x\in\Omega.$$ Hence, in this case $u^*$ satisfies the Pontryagin maximum principle.

\subsection{The proximal gradient method with variable stepsize}

The convergence results of this section require the knowledge of the Lipschitz modulus $L_f$ of $\nabla f$.
This can be overcome by line-search with respect to the parameter $L$ subject to a suitable decrease condition, which is a widely applied technique.

\begin{algorithm} [Proximal gradient with  variable step-size]
	\label{alg:PG-V} Choose $\eta>0$ and $ u_0\in U_{ad}$. Set $k=0$.
	\begin{enumerate}
		\item Determine $L_{k}\geq 0$ and $u_{k+1}$ as global  solution of
		\begin{equation*}\label{ch4:subP-PGV} \min\limits_{u\in  L^2(\Omega)}  f(u_k)+\nabla f(u_k)(u-u_k)+\frac{L_{k}}{2}\|u-u_k\|^2_{L^2(\Omega)}+j(u)\end{equation*}
		such that
		\begin{equation}\label{alg:algdecrease_cond}
		\eta\|u_{k+1}-u_k\|^2_{L^2(\Omega)}\leq (f(u_k)+j(u_k))-(f(u_{k+1}+j(u_{k+1}))\end{equation} is satisfied.
		\item Set $k:=k+1$, repeat.
\end{enumerate} \end{algorithm}

The convergence results as in Theorem \ref{thm:conv} can be carried over.
Then theorem \ref{thm:conv} holds without the assumption $L>L_f$.
The assumptions $1/L>s_0$ has to be replaced by $(\limsup L_k)^{-1}>s_0$.
This is satisfied if $s_0=0$, which is true by Theorem \ref{thm:discontinuity} if one of \ref{assb:4b}, \ref{assb:4c} is valid.

\section{Applications of the proximal gradient method}\label{sec5}
\subsection{Optimal control with $L^p$ control cost, $p\in(0,1)$}\label{ch:4}
In \cite{DW:1}, the discussed proximal method was analyzed and applied to optimal control problems with $L^0$ control cost, i.e., $g(u):=\frac\alpha2u^2+|u|_0$. In this section, we discuss the problem with $g(u):=\frac\alpha2 u^2+\beta|u|^p+\delta_{[-b,b]}$, where $p\in(0,1)$ and $b\in(0,\infty]$ and
consider
 \begin{equation} \label{P:Lp}
\min\limits_{u\in L^2(\Omega)}f(u)+\frac{\alpha}{2}\|u\|_{L^2(\Omega)}+\beta\int_\Omega|u(x)|^p\dx
\end{equation}
s.t. $$u\in U_{ad}:=\{u\in L^2{(\Omega)}:\:| u(x)|\leq b\:\text{ a.e. in }\Omega\}$$
 with $\alpha\geq0,\:\beta>0$.

To find a solution  to \eqref{P:Lp}with Algorithm \ref{alg:PG}, the subproblem, interpreted in terms of \eqref{P_alpha} with $\tilde g:= |u|^p+\delta_{[-b,b]},$
\begin{equation*}\label{ch4:subP1} \min\limits_{u\in U_{ad}}  f(u_k)+\nabla f(u_k)(u-u_k)+\frac{L}{2}\|u-u_k\|^2_{L^2(\Omega)}+\frac{\alpha}{2}\|u\|_{L^2(\Omega)}+\beta\int_\Omega|u(x)|^p\dx\end{equation*}
has to be solved in every iteration. According to Theorem \ref{lem:infdimsetting}, $u_{k+1}$ is a solution to \eqref{ch4:subP1} if and only if
\begin{equation*}\label{ch4:cond} u_{k+1}(x)\in \prox_{\frac{\beta}{L+\alpha}\tilde g}\left(\frac{1}{L+\alpha}(Lu_k(x) -\nabla f(u_k)(x))\right)\enspace f.a.a. \enspace x\in\Omega.\end{equation*}
Due to Theorem \ref{thm:discontinuity} it holds $u_{k+1}(x) =0$ or $|u_{k+1}(x)|\geq u_0$ for all $k$. The particular choice of $g$ allows to compute the constant $u_0$ explicitly by solving $\min\limits_{u\ne0}\frac{u}{2}+s\frac{g(u)}{2}$ and is given by $$u_0\left(\frac{\beta}{\alpha+L}\right) = \min\left(b, \left(\frac{\alpha+L}{2\beta(1-p)}\right)^{\frac{1}{p-2}}\right)$$
as a consequence of Lemma \ref{lem:ass_q0}.

\begin{figure} [H]
	\centering
	\includegraphics[height=6cm]{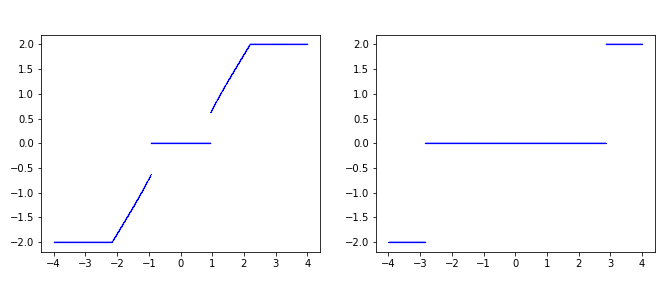}
	\caption{The mapping $\prox_{sg}(q)$ for parameters $(s,b,p) = (0.5,2,0.5)$ (left) and $(s,b,p) = (3,2,0.3)$ (right) with $\tilde g(u):=|u|^p+\delta_{[-b,b]}$.}
	\label{ch4:figH}
\end{figure}

We recall the definition of the set-valued map $\G:\R\to\R$, which reads in this case
$$u\in\G(z):=\G_{L,\alpha,s}:\Longleftrightarrow u=\argmin\limits_{|v|\leq b}-zv+\frac L2(u-v)^2+\frac\alpha2v^2+s|v|^p.$$
Note that $g$ satisfies assumptions \ref{assb:6} and \ref{assb:7} due to its structure. This allows to give an equivalent but more precise characterization of $\G$  as Lemma \ref{lem328_} applies to $u_{k+1}(x)$ on $I_{k+1}$.

\begin{corollary} Let $u\geq u_0({\frac{\beta}{L+\alpha}})$. Then
	$u\in\G(z_k(x)) \Longleftrightarrow u$ is a stationary point of \begin{equation*}\label{zProb} \min\limits_{u:|u|\leq b} -z_k(x)u+\frac{\alpha}{2}u^2+{\beta}|u|^p
	\end{equation*} for almost all $x\in I_{k+1}$.
\end{corollary}

A visualization of $\G$ is given in Figure \ref{fig:mapG} below.
\begin{figure} [H]
\centering{
\includegraphics[height=6cm]{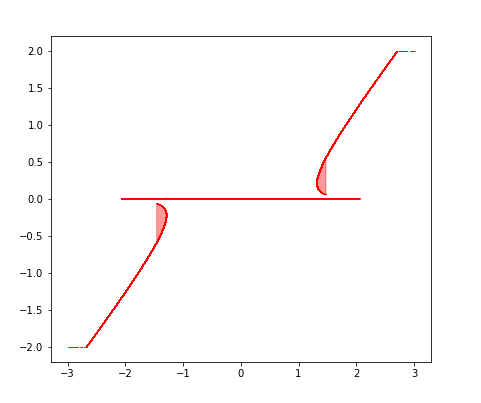}}
\caption{The union $(\mathcal G_0\cup\mathcal G^+\cup\G^-)(q)$ and the convexified map $(\mathcal G_0\cup\overline{\conv}\mathcal G^+\cup\overline{\conv}\G^-)(q)$ (filled area) (right)  for parameters $(L,\alpha,s,b) = (0.1,0.01,0.01,2)$ and $\tilde g(u):=|u|^{0.8}+\delta_{[-b,b]}$.
}
\label{fig:mapG}
\end{figure}
With a suitable choice of parameters, we can apply Theorem \ref{thm:strongconv} to the $L^p$ problem to obtain a strong convergent subsequence.

\begin{corollary}
\label{thm:strongconvLp}
Let $\alpha >0$ and $(u_k)$ a sequence of iterates. Furthermore, assume $L\leq (\frac{2}{p}-1)\alpha$. Then the assumptions of Theorem \ref{thm:strongconv} are satisfied. If in addition $\nabla f$ is completely continuous from $L^2(\Omega)$ to $L^2(\Omega)$, then every weak sequential limit point $u^*\in L^2(\Omega)$ is a strong sequential limit point in $L^1(\Omega)$.
\end{corollary}
\begin{proof}
	Let $k\in\mathbb{N}$. It holds $|u_{k+1}(x)| \geq u_0$ with $u_0 :=  \min\left(b,\left(\frac{\alpha+L}{2\beta(1-p)}\right)^{\frac{1}{p-2}}\right)$ on $I_{k+1}$. A short calculation yields that the assumptions on the parameters imply
	$$\left(\frac{\alpha+L}{2\beta(1-p)}\right)^{\frac{1}{p-2}} \geq \left(\frac{\alpha}{\beta p(1-p)}\right)^{\frac{1}{p-2}} =: u_{I}.$$
	Here, $u_{I}$ is the positive point of inflection of \eqref{zProb} and it holds that $$h_{q,\frac{\beta}{\alpha}}(u) = -qu+\frac{1}{2}u^2+\frac{\beta}{\alpha}|u|^p$$ is convex for all $q\in\R$ on $[u_{I},\infty)$ and $(-\infty, u_{I})$, respectively, which corresponds to Assumption \ref{assb:7}. The claim now follows by Lemma \ref{lem:330} and Theorem \ref{thm:strongconv}.
\end{proof}

\color{black}
\subsection{Optimal control with discrete-valued controls}
Let us investigate the optimization problem with optimal control taking discrete values.
That is, we choose $g(u)$ as the indicator function of integers, i.e.,
\[
g(u):=\delta_\mathbb{Z}(u) := \begin{cases}
0\quad &\text{if  } u\in \mathbb{Z},\\
\infty &\text{else}
\end{cases}.
\]
The problem now reads
\begin{equation} \label{P_discr}
\min_{u\in L^2(\Omega)} f(u)+ \int_\Omega\delta_\mathbb{Z}(u(x))\dx.
\end{equation}
Note, this choice satisfies Assumption \ref{assb:4c}. Applying Algorithm \ref{alg:PG}, the subproblem to solve is given by

\begin{equation}\label{ch4:subP_discr} \min\limits_{u\in L^2(\Omega)}  f(u_k)+\nabla f(u_k)(u-u_k)+\frac{L}{2}\|u-u_k\|^2_{L^2(\Omega)}+\int_\Omega \delta_\mathbb{Z}(u(x))\dx
\end{equation}
and can be solved pointwise and explicitly. The analysis carried out in Chapter \ref{section:PG} is applicable, however, the special choice of $g$ comes along with the following desirable result.
\begin{lemma} \label{lem:discr_estimate}
	Let $u_k,u_{k+1}\in U_{ad}$ be consecutive iterates of Algorithm \ref{alg:PG}. Then $$\|u_{k+1}-u_k\|^p_{L^p(\Omega)}\geq\|u_{k+1}-u_k\|_{L^1(\Omega)}$$
	holds for all $p\in[1,\infty)$.
\end{lemma}
\begin{proof}
	The claim follows directly, since either $|u_{k+1}(x)-u_k(x)|= 0$  or $|u_{k+1}(x)-u_k(x)|\geq 1$ as the iterates are integer-valued in almost every point.
\end{proof}
Lemma \ref{lem:discr_estimate} implies strong convergence of iterates $(u_k)$ in $L^1(\Omega)$.
\begin{theorem}
	Let $(u_k)$ be a sequence generated by Algorithm \ref{alg:PG} with weak limit point $u^*$. Then $u_k\to u^*$ in $L^1(\Omega)$.
\end{theorem}
\begin{proof} 
%
%
As in the proof of Theorem \ref{thm:conv}, we get
	$$\sum\limits_{k=1}^\infty\|u_{k+1}-u_k\|^2_{L^2(\Omega)}<\infty$$ and therefore by Lemma \ref{lem:discr_estimate}
	$$\sum\limits_{k=1}^\infty\|u_{k+1}-u_k\|_{L^1(\Omega)}\leq \sum\limits_{k=1}^\infty\|u_{k+1}-u_k\|^2_{L^2(\Omega)}<\infty$$
	Thus, $(u_k)$ is a Cauchy sequence in $L^1(\Omega)$ and therefore convergent in  $L^1(\Omega)$ and it holds $u_k\to u^*$.
\end{proof}

\section{Numerical experiments}
In this section we finally apply  the proximal gradient method to  optimal control problems of type \eqref{P} and carry out numerical experiments for cost functionals with different $g$.

Let in the following denote $f_l$ the reduced tracking-type functional $$f_l(u):=\|S_lu-y_d\|_{L^2(\Omega)}^2,$$ where $S_l$ is the weak solution operator of the linear Poisson equation \begin{equation}\label{ch4:numPois}-\Delta y = u\enspace \text{ in }\Omega,\enspace y=0\text{ on }\partial\Omega.\end{equation}
Further we define the nonlinear solution operator $S_{sl}$ of the semilinear equation
\begin{equation}\label{ch4:numPois_semlin}-\Delta y +d(y)= u\enspace \text{ in }\Omega,\enspace y=0\text{ on }\partial\Omega\end{equation}
where $d(x,y):\Omega\times\R\to\R$  is a $C^2$ Carathéodory Function  with respect to $y$
 with  $d(\cdot,0)$ in $L^p(\Omega)$, $n<p$, satisfying

 \begin{enumerate}
 	\item $\displaystyle {\frac{\partial d}{\partial y}(x,y)\ge0}$ for almost all  $x\in\Omega$,
 	\item $\forall \:M>0\:\exists C_M>0$ s.t.
 	$\displaystyle{\left|\frac{\partial d(x,y)}{\partial y}\right|+\left|\frac{\partial^2d(x,y)}{\partial y^2 }\right|\le C_M}$ for almost all $x\in\Omega$ and $|y|\le M$.

 \end{enumerate}

Then the equation is uniquely solvable, we refer to e.g., \cite{CasasL12012,Casas1993}
In addition, we define
$$f_{sl}:=\|S_{sl}(u)-y_d\|_{L^2(\Omega)}^2.$$
Furthermore, we choose $\Omega:=(0,1)^2$ to be the underlying domain in all following examples.
To solve the partial differential equation, the domain is divided into a regular triangular mesh and the PDE \eqref{ch4:numPois},\eqref{ch4:numPois_semlin} is discretized with piecewise linear finite elements. The controls are discretized with piecewise constant functions on the triangles. The finite-element  matrices were created with FEnicCS {\cite {fenicstutorial:1}}. If not mentioned otherwise, the meshsize is approximately $h =\sqrt{2}/160 \approx 0.00884$.
In each iteration a suitable constant $L_k>0$ needs to be determined, that satisfies the decrease condition \begin{equation}
\eta\|u_{k+1}-u_k\|^2_{L^2(\Omega)}\leq (f(u_k)+j(u_k))-(f(u_{k+1}+j(u_{k+1})),\end{equation} see \eqref{alg:algdecrease_cond}. Note, $L_k^{-1}$ can be seen as a stepsize. In \cite{DW:1} several stepsize selection strategies are proposed. In our tests, we use a simple Armijo-like backtracking line search method (\textbf{BT}).  That is, having an initial $L^0>0$ and a widening factor $\theta\in(0,1)$, determine $L_k$ as the smallest accepted  number of form $L^0\theta^{-i},\enspace i =0,1,...$. This method ensures a decrease in the objective values along the iterates, but it turns out to be very slow for large $L_0$, as the corresponding stepsize $L_k^{-1}$ gets smaller. For all our tests we choose
$$\eta = 10^{-4},\quad \theta = 0.5.$$
The stopping criterion is as follows:
\begin{quote} If $|f(u_{k+1})+g(u_{k+1})-(f(u_k)+g(u_k)|\leq 10^{-12}$:\\ \hspace*{0.8cm}STOP. \end{quote}

First, we consider control problems with $L^p$ control cost, which were investigated in chapter \ref{ch:4},  i.e., $g(u):=|u|^p+\delta_{[-b,b]}$ with $p\in(0,1)$.\paragraph{ Example 1} Let $g(u):=|u|^p+\delta_{[-b,b]}$ for $p\in(0,1)$ and find $$\min_{u\in L^2(\Omega)}f_l(u)+\|u\|_{L^2(\Omega)}^2+\beta\int_\Omega g(u(x))\dx.$$
Setting $U_{ad}:=\{L^2(\Omega):|u(x)|\leq b \text{ a.e. on }\Omega\}$ the problem is equivalent to
$$\min_{u\in U_{ad}}f_l(u)+\|u\|_{L^2(\Omega)}^2+\beta\int_\Omega |u(x)|^p\dx.$$
The first example is taken from {\cite{DW:1}}, where the proximal gradient algorithm was investigated for (sparse) optimal control problems with $L^0(\Omega)$ control cost.
Since $\int_\Omega|u|^pdx\to\int_\Omega|u|^0dx$ as $p\searrow0$, we expect similar solutions. We choose the same problem data as in  \cite{DW:1,ItoKunisch:1}. That is, if not mentioned otherwise, $$y_d(x,y) = 10x\sin(5x)\cos(7y)$$ and $\alpha = 0.01,\enspace \beta = 0.01, \enspace b = 4.$

A computed solution for $p=0.8$ is shown in Figure \ref{ch4:fig:lös}.
\begin{figure} [H]
	\centering

	\includegraphics[height=6cm]{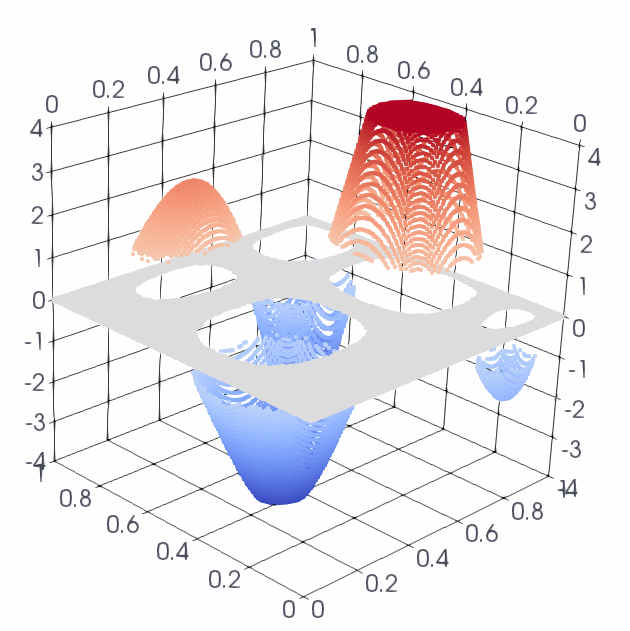}
	\caption{Solution $u$}

	\label{ch4:fig:lös}
\end{figure}

\paragraph{Convergence for decreasing  $p-$values.}
In the following we consider solutions for different values of $p$. We use the same data and discretization as above. We set $L_0 = 0.0001$.
\begin{table}[htbp]
	\centering
	\begin{tabular}{cccc}
		\hline
		$p$ & $J(u^*)$ & $N_p(u^*)$ &no. pde \\
		\hline
		0.5 & 5.3831 & 0.6711 & 15\\
		0.3 & 5.3819 & 0.5725 & 15 \\
		0.1 & 5.3808 & 0.4841 & 15\\
		0.01 & 5.3804& 0.4482 & 15\\
		0.001 & 5.3804 & 0.4448 & 15\\
		\hline
		0 & 5.38034 & 0.4445 & 15 \\
		\hline
	\end{tabular}
	\caption{Decreasing values of $p$}
	\label{table:T3}
\end{table}
In Table \ref{table:T3} it can be seen that $J(u^*)$ and $\displaystyle\int_\Omega|u^*|^pdx$ converge for decreasing values of $p$. The last row in Table \ref{table:T3} shows the result of applying the iterative hard-thresholding algorithm IHT-LS from  \cite{DW:1}  to the problem with $p=0$, which is in agreement with our expectation.  In the implementation we used a meshsize of $h = \sqrt{2}/500 \approx  0.0028$.

\paragraph{Discretization.}
Next, we solved the problem on different levels of discretization to investigate the influence. As can be seen in Table \ref{table:T4} the algorithm stays robust across different mesh sizes.\\

\begin{table}[htbp]
	\centering
	\begin{tabular}{cccc}
		\hline
		$h$ & $J(u^*)$ & $N_p(u^*)$ &no. pde \\
		\hline
		0.071 & 5.2239 & 0.6371 & 13\\
		0.035 & 5.3429 & 0.6581 & 15 \\
		0.0177 & 5.3732 & 0.6686 & 15 \\
		0.00884 & 5.3808 & 0.6704 & 15\\
		0.00442 & 5.3827 & 0.6710 & 15\\
		0.00221 & 5.3832 & 0.6711 & 15\\
		\hline
	\end{tabular}
	\caption{influence of meshsize}
	\label{table:T4}
\end{table}

\paragraph{Convergence in the case $\boldsymbol{L> (2/p-1)\alpha}$.} So far, in every experiment the assumption on the parameters was naturally satisfied, such that strong convergence of iterates can be proven according to Theorem \ref{thm:strongconvLp}. The numerical results confirmed the theory. We will now investigate the case where the assumption is not satisfied, i.e., we choose parameters such that $L> (2/p-1)\alpha$.  In the following we present the result for the problem parameters $$ \alpha = 0.001,\enspace p = 0.9, \enspace L_0= 0.005.$$
Furthermore, we set $b = 6$.
In our computations the algorithm needed very long to reach the stopping criteria $|J(u_{k+1})-J(u_k)|\leq 10^{-12}$ as can be seen in Table \ref{table:T5}.
This might be due to the parameter choice and the step-size strategy. For smaller mesh-sizes more iterations are needed.
\begin{table}[htbp]
	\centering
	\begin{tabular}{cccc}
		\hline
		$h$ & $J(u^*)$ & $N_p(u^*)$ &no. pde \\
		\hline
		0.00884 & 5.3567 & 1.1246 & 395\\
		0.00442 & 5.3567 & 1.1247 & 601\\
		0.00221 & 5.3567 & 1.1253 & 821\\
		\hline
	\end{tabular}
	\caption{performance for bad choice of parameters across different mesh-sizes}
	\label{table:T5}
\end{table}

Recall, the problem in the analysis that comes with this choice of parameters is that the map $\mathcal G$ in Lemma \ref{lem:stat_points} is not necessarily single-valued anymore on the set of points where an iterate is not vanishing, see also Figure \ref{fig:mapG}.
Let $u_I:=u_{I}(\beta/\alpha)>0$ denote the  constant from Assumption \ref{assb:7}
and define the set
$$\Omega_{m,k}:=\{x\in\Omega: 0<|u_k(x)|<u_{I}\}.$$
Then $\Omega_{m,k}$ is the set of points for which the crucial assumption in Lemma \ref{lem:330} that implies single-valuedness of $\G\setminus\{0\}$ is not satisfied.
In our numerical experiments, however, we made the observation that the measure of the set $\Omega_{m,k}$ is decreasing as $k\to\infty$, see Figure \ref{fig:ch4measure}.
Across different mesh-sizes h, the measure decreases and tends to zero along the iterations.
\begin{figure} [H]
	\centering

	\includegraphics[height=6cm]{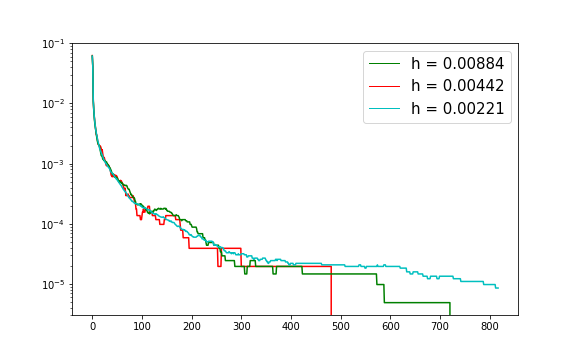}
	\caption{Measure of $\Omega_{m,k}$ at iteration $k$ for different discretization levels}

	\label{fig:ch4measure}
\end{figure}

Unfortunately, we were not able to prove such a behavior in the analysis and have no theoretical evidence whether this can be expected in general. But assuming $$ |\Omega_{m,k}|\to0$$ based on our numerical result,   strong convergence of the sequence $(u_k)$ can be concluded similar to Theorem \ref{thm:strongconv}.

\paragraph{Example 2} Let us now consider the semilinear problem
$$\min_{u\in U_{ad}}f_{sl}(u)+\|u\|_{L^2(\Omega)}^2+\beta\int_\Omega g(u(x))\dx$$ with $g(u)= |u|^p$, $p\in(0,1)$. This example can be found in \cite{CasasL12012} for semilinear control problems with $L^1$-cost. Here, $f_{sl}$ is given by the standard tracking type functional 
$u\mapsto\|y_u-y_d\|^2_{L^2(\Omega)}$, where $y_u$ is the solution of the semilinear elliptic state equation
\[
-\Delta y+y^3=u \quad\text{in }\Omega,\quad y = 0\quad\text{on }\partial\Omega.
\]
The data is given by $\alpha = 0.002$, $\beta =0.03$, $b = 12$ and $y_d = 4\sin(2\pi x_1)\sin(\pi x_2)e^{x_1}$. We use the parameter $L_0 = 0.001$.
\begin{figure} [H]
	\centering

	\includegraphics[height=6cm]{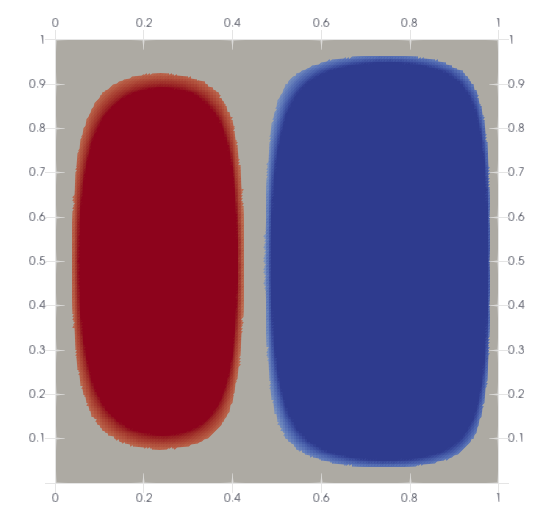}
	\caption{solution $u$ of the semilinear optimal control problem with $g(u):=|u|^{0.5}.$}

	\label{fig:lsg_semlin}
\end{figure}
We made similar observations as in the linear case concerning the influence of discretization and different values of $p$. Also the behavior of the algorithm in case of a bad choice of parameters is as before (see Example 1).

\paragraph{Example 3} In this last test, we consider an optimal control problem with discrete-valued controls. That is, we choose $$g(u) := \delta_\mathbb{Z}(u),$$
where $\delta_M$ denotes the indicator function of a set $M$, i.e., $\delta_M(u):=\begin{cases}
0\quad &\text{if  } u\in M,\\
\infty &\text{else}
\end{cases}$.
Here, the subproblem in Algorithm $\ref{alg:PG}$ can be solved pointwise and explicitly.
We adapt again the setting from Example 1. In Figure \ref{fig:lsg_disc}, a solution plot of the optimal control is displayed. We used exactly the same problem data as before in Example 1, but set $b=2$ and $L_0 = 0.001$.
Again, we find the algorithm is robust with respect to the discretization.
\begin{figure} [H]
	\centering

	\includegraphics[height=6cm]{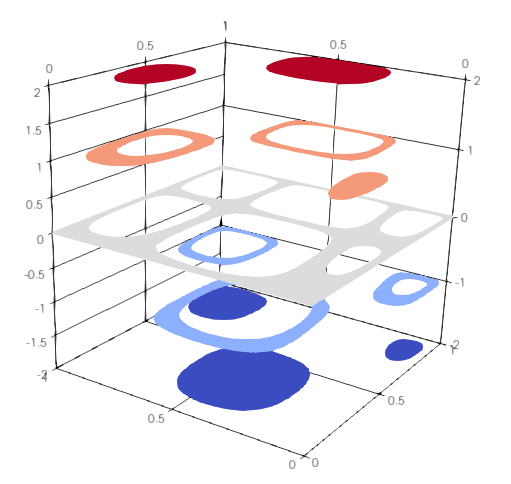}
	\caption{optimal control with discrete values}

	\label{fig:lsg_disc}
\end{figure}


\bibliography{literature}
\bibliographystyle{plain_abbrv}

\end{document}